\documentclass[11pt]{article}
\usepackage{a4wide}
\usepackage{amsmath, amsthm, xcolor}
\usepackage{amssymb}
\usepackage{graphicx,soul}
\usepackage[normalem]{ulem}
\usepackage{relsize}

\newcommand{\hidden}[1]{}

\usepackage{color}
\usepackage[mathscr]{euscript}
\usepackage{tikz}

\newtheorem{lemma}{Lemma}[section]
\newtheorem{problem}{Problem}[section]

\theoremstyle{definition}

\newtheorem{definition}{Definition}[section]
\newtheorem{remark}{Remark}[section]
\numberwithin{equation}{section}

\theoremstyle{plain}
\newtheorem{thm}{Theorem}[section]

\newtheorem{cor}{Corollary}[section]

\newtheorem{prop}{Proposition}[section]

\newcommand{\Z}{\mathbb{Z}}
\newcommand{\Q}{\mathbb{Q}}
\newcommand{\R}{\mathbb{R}}

\newcommand{\N}{\mathbb{N}}

\newcommand{\RR}{\mathbb{R}}

\newcommand{\NN}{\mathbb{N}}

\newcommand{\ZZ}{\mathbb{Z}}
\newcommand{\PP}{\mathbb{P}}

\newcommand{\cW}{\mathcal{W}}
\newcommand{\cC}{\mathcal{C}}

\newcommand{\sC}{\mathscr{C}}

\newcommand{\ba}{\mathbf{a}}
\newcommand{\br}{\mathbf{r}}
\newcommand{\bs}{\mathbf{s}}
\newcommand{\bx}{\mathbf{x}}
\newcommand{\by}{\mathbf{y}}

\newcommand{\bu}{\mathbf{u}}
\newcommand{\be}{\mathbf{e}}

\newcommand{\bv}{\mathbf{v}}

\newcommand{\bp}{\mathbf{p}}
\newcommand{\bP}{\mathbf{P}}

\newcommand{\bL}{\mathbf{L}}
\newcommand{\bZ}{\mathbf{Z}}
\newcommand{\bY}{\mathbf{Y}}

\newcommand{\bq}{\mathbf{q}}
\newcommand{\bzero}{\mathbf{0}}

\newcommand{\Sym}{\mathrm{Sym}}

\newcommand{\SL}{\mathrm{SL}}

\newcommand{\bX}{\mathbf{X}}

\newcommand{\vx}{\mathbf{x}}

\DeclareMathOperator{\e}{e}

\newcommand{\eps}{\varepsilon}
\newcommand{\DI}{\mathbf{DI}}
\newcommand{\Sing}{\mathbf{Sing}}
\newcommand{\VSing}{\mathbf{VSing}}
\newcommand{\Bad}{\mathbf{Bad}}
\newcommand{\bfW}{\mathbf{W}}
\newcommand{\FS}{\mathbf{FS}}
\newcommand{\bfE}{\mathbf{E}}

\newcommand{\bT}{\mathbf{T}}
\newcommand{\bG}{\mathbf{G}}
\newcommand{\bH}{\mathbf{H}}
\newcommand{\Spec}{\mathbf{Spec}}

\newcommand {\ignore}[1]  {}

\newif\ifdraft\drafttrue

\draftfalse

\newcommand{\btheta}{{\bf x}}

\newcommand{\x}{{\bf x}}

\renewcommand{\emptyset}{\varnothing}
\renewcommand{\setminus}{\smallsetminus}

\newcommand{\ourappendix}{the last section}

\begin{document}


\title{ Dirichlet is not just    Bad  \\ and Singular}

\author{ Victor Beresnevich\footnote{Research partly supported by EPSRC Programme grant EP/J018260/1} \\ {\small\sc (York) }   \and  Lifan Guan\footnote{Research supported by EPSRC Programme grant EP/J018260/1 and ERC Consolidator grant 648329}  \\ {\small\sc (G\"ottingen) } \and Antoine Marnat\footnote{Research supported by EPSRC Programme grant EP/J018260/1 and FWF Project I 3466-N35} \\ {\small\sc (TU Graz) } \and  ~ \and
 Felipe Ramirez\footnote{Research partly supported by EPSRC Programme grant EP/J018260/1} \\ {\small\sc (Wesleyan)}  \and
Sanju Velani\footnote{Research partly supported by EPSRC Programme grant EP/J018260/1}\ \\ {\small\sc (York) }}

\date{\emph{}}


\maketitle

\centerline{{\it Dedicated to Pru and Chim - both 60no and going strong!}}

\bigskip

\begin{abstract}
It is well known that in dimension one the set of Dirichlet improvable real numbers consists  precisely of badly approximable and singular numbers.  We show that in higher dimensions this is not the case by proving that there exist continuum many  Dirichlet improvable vectors that are neither badly approximable nor singular.   This is a consequence of a stronger statement {that involves} {very} well approximable {points}. {In  \ourappendix{} we formulate the notion of intermediate Dirichlet improvable sets {concerning approximations by rational planes of every intermediate dimension} and show that they coincide.  This naturally extends a classical theorem of Davenport $\&$ Schmidt (1969) which states that the simultaneous form of Dirichlet's theorem is improvable if and only if the dual form is improvable. {Consequently,} our {main} ``continuum'' result is equally valid for the corresponding intermediate Diophantine sets of badly approximable, singular and Dircihlet improvable points.}
\end{abstract}


\renewcommand{\baselinestretch}{1}

\parskip=1ex

\maketitle

\section{Introduction}

\subsection{Background and motivation \label{BM}}

The {main} goal of this paper is to investigate the relation between three basic sets arising from  Dirichlet's fundamental theorem in the classical theory of Diophantine approximation.  It is therefore natural to start with the statement of the theorem and in turn describe the associated sets.  For $ x \in \RR$, let $ \langle x\rangle := \min \{ | x-m | : m \in \ZZ \} $
denote the distance from $x$ to the nearest integer and for $ \vx=(x_1, \ldots,x_n) \in \RR^n$ let
$$
\langle \vx\rangle := \max_{1 \le i \le n} \langle  x_i\rangle \, .
$$

\begin{thm}[Dirichlet\index{Dirichlet's theorem}]\label{Dir}
 For any $\vx \in \RR^n$ and
  $N \in \mathbb{N}$, there exists $q \in \mathbb{Z}$ such that
  \begin{equation}\label{h7.0}
  \langle q \vx\rangle  < N^{-\frac1n} \qquad { and }  \qquad  1 \le  q \le N \, .
  \end{equation}
\end{thm}

  An important consequence of Dirichlet's theorem
is the following statement.

\begin{cor}[Dirichlet]\label{CorDir}
 For any $\vx \in \RR^n$ there exist infinitely many $q \in \mathbb{N}$ such that
  \begin{equation}\label{corh7.0}
   \langle  q \vx\rangle  < q^{-\frac1n}    \, .
  \end{equation}
\end{cor}

\noindent  The above foundational theorem from the theory  of simultaneous Diophantine approximation
prompts the following natural question:

\noindent\textbf{Question I.} Can Dirichlet's theorem be improved?

\medskip

\noindent  The following notions help make the question more precise.  Following Davenport $\&$ Schmidt \cite{DavSchI},  for a particular $\vx \in \RR^n$ we say that improvement in Dirichlet's theorem is possible if there exists a constant $\eps \in (0,1)$ such that, for all  $N > N_0( \bx, \eps)$   sufficiently large there exists $q \in \mathbb{Z}$ such that
  \begin{equation}\label{DIsv1}
  \langle q \vx\rangle  < \eps N^{-\frac1n} \qquad { and }  \qquad  1 \le  q \le N \, .
  \end{equation}
For obvious reasons such an $\vx$ is referred to as  \emph{Dirichlet improvable} and we let $\DI_n$ denote the set of Dirichlet improvable points in $\R^n$.  Furthermore, we say that  $\vx \in \R^n$  is \emph{singular} if it is Dirichlet improvable with $\eps> 0$ arbitrarily small; that is, for  any  $\eps> 0$ for all sufficiently
large $ N $ there exists $q \in \mathbb{Z}$  satisfying \eqref{DIsv1}.  We let $\Sing_n$ denote the set of singular points in  $\R^n$. By definition, we  clearly  have that
$$
\Sing_n \subseteq \DI_n    \,
$$
 and it is easily verified   that $\Sing_n$ contains every rational hyperplane in $\R^n$. Thus
$$
n-1 \le  \dim \Sing_n \le n   \, .
$$
Here and throughout,  $\dim X$ will denote the Hausdorff dimension of a  subset $X$ of  $\R^n$.  In the case $n=1$, a nifty argument due to Khintchine \cite{Khin1} dating back to the twenties shows that a real number is singular if and only if it is rational; that is
\begin{equation} \label{kh1}\Sing_1 = \Q \, .
 \end{equation}
Recently, Cheung $\&$ Chevallier \cite{ChCh}, building on the spectacular $n=2$ work of Cheung \cite{Ch}, have shown that
$$
 \dim \Sing_n =  \frac{n^2}{n+1}     \qquad {\text{for all }n \ge 2} \, .
$$
Note that since $ \frac{n^2}{n+1} > n-1 $,  this immediately implies that in higher dimensions $ \Sing_n $ does not simply correspond to rationally dependent $\vx \in \R^n$ as in the one-dimensional case -- the theory is much richer.

In \cite{DavSchI}, Davenport $\&$ Schmidt established various results concerning the set $\DI_n$ of Dirichlet improvable points and the set $\Bad_n$ of simultaneously  badly approximable points.  In particular they showed (see  \cite[Theorem 2]{DavSchI}) that
\begin{equation} \label{ds1}
 \Bad_n  \ \subseteq  \  \DI_n \, .
\end{equation}
 Recall,  $\vx \in \R^n$ is said to be {\em badly approximable} if there
exists a constant  $\eps=\eps(\bx)  \in (0,1)$ so that
  \begin{equation}\label{badsv1}
  \langle q \vx\rangle  > \eps q^{-\frac1n} \qquad { \forall}  \quad  q \in \mathbb{N}\, .
  \end{equation}
In other words,  $\Bad_n$  corresponds to those $\vx \in \R^n$ for which the right hand side of the inequality appearing in Dirichlet's corollary, namely \eqref{corh7.0}, cannot be improved by  $\eps > 0 $ arbitrarily small.  By definition, we clearly have that $\Bad_n \cap \Sing_n = \emptyset$.  It is worth mentioning the well known fact that  $\Bad_n$ is a set of $n$-dimensional Lebesgue  measure zero but of full dimension; i.e.
$$
\dim \Bad_n = n \, .
$$
In view of  \eqref{ds1} it thus  follows that
$$
\dim \DI_n = n \, .
$$
In a follow-up paper \cite{DavSch}, Davenport $\&$ Schmidt   showed that $\DI_n $ is a set of $n$-dimensional Lebesgue measure zero and thus in terms of measure and dimension it has the same properties as the set  $\Bad_n$. In the case $n=1$, much more is true: any irrational $x \in \R$ is Dirichlet improvable if and only if it is badly approximable.  This for example follows directly from \cite[Theorem 1]{DavSchI} and together with  \eqref{kh1} implies that
\begin{equation}\label{ohyes1}
   \DI_1  \  =   \  \Bad_1  \ \cup \  \Sing_1  \, .
\end{equation}
{Thus,} in view of \eqref{ohyes1} we have a complete characterisation  in dimension one. In higher dimensions, we know that
$$
\DI_n  \  \ \supseteq \  \Bad_n  \ \cup \  \Sing_n  \,
$$
but surprisingly it seems unknown whether or not equality is possible.  In other words, the answer to the following basic problem seems unknown. As far as we are aware, it first appeared in print in Fabian S\"{u}ess' beautifully written PhD thesis \cite[Section~4.1]{FSthesis}.

\begin{problem}\label{prob-sanju}
 Is the set $ \ \DI_n\setminus (\Bad_n\cup \Sing_n)$ empty when $n\ge 2$?
\end{problem}

\noindent The {key goal} of this paper is to show that it is not.  Maybe it is ``folklore'' that in higher dimensions there exist Dirichlet improvable points that are neither badly approximable nor singular. However,  we would like to stress that we are unaware of any such  a statement.

\begin{thm} \label{thslv1}  For $n \ge 2$,  the set
$$
\FS_n := \DI_n  \setminus  (\Bad_n  \ \cup \  \Sing_n )
$$
has continuum many points.
\end{thm}


We suspect  that our theorem is far from the truth. Indeed,  it may well be the case  that for $ n \ge 2$
$$
\dim \FS_n = n  \, .
$$
As we shall see in the next section, we actually prove a more general and effective version of Theorem \ref{thslv1}.  Unfortunately, it sheds no light on the dimension of $ \FS_n$.


\begin{remark} Following the appearance of the first pre-print version of this paper on the arXiv, Nikolay Moshchevitin kindly pointed out the work of Akhunzhanov and Shatskov \cite{AS}.  For $n=2$, they compute the Dirichlet's spectrum for simultaneous approximation by rational points with respect to the Euclidean norm. In short, their method uses the theory of best approximations and can be adapted to construct Dirichlet improvable points in $\R^2$ that are not simultaneously singular or badly approximable.
\end{remark}

\subsection{The setup, further background and main results}
\label{setupmain}

Recall that from the classical point of view there are two forms of Diophantine approximation in $\R^n$; one corresponding to  (simultaneous) approximation by rational points as considered
in the previous section and the other corresponding to (dual) approximation by rational hyperplanes.   Concerning the latter, the dual version of Dirichlet's theorem states that  for any $\bx \in \R^n$ and
  $N \in \mathbb{N}$ there exists $\bq \in \mathbb{Z}^n{\setminus}\{0\}$
   such that
  \begin{equation}\label{slv-dual}
\langle \bq \cdot \bx \rangle \le  N^{-n} \qquad { and }  \qquad \| \bq \| \le N . \end{equation}
Here and elsewhere  $\bq \cdot \bx := q_1x_1 + \dots +q_nx_n $ is the standard inner product and $\| \bq \|:=\max \{|q_1|, \ldots, |q_n|\}$ is the supremum norm of $\bq$.  In \cite[Theorem~2]{DavSch}, Davenport $\&$ Schmidt proved that the dual form of Dirichlet's theorem is improvable if and only if the simultaneous form of Dirichlet's theorem (Theorem \ref{Dir}) is improvable. So it follows that $\bx\in\DI_n$ if and only if there exists $\eps\in (0,1)$ such that, for all $N > N_0( \bx, \eps)$ sufficiently large there exists $\bq \in \mathbb{Z}^n{\setminus}\{0\}$ such that
\begin{equation}\label{equ-basic}
\langle \bq \cdot \bx \rangle \le  \eps N^{-n} \qquad { and }  \qquad \| \bq \| \le N .
\end{equation}
Indeed, the same transference between the simultaneous and dual forms of approximation is true when considering the set of singular points $\Sing_n$.  However, this dual versus simultaneous equivalence for singular points  (and indeed badly approximable points) holds in a much wider context.  This we now describe since it will be the setting of our main  result.

Let $d$ be integer satisfying $0 \le d \le n-1$.
The setup we now consider is one in which we approximate points $ \bx \in \R^n$ by
$d$-dimensional rational affine subspaces $L \subset \R^n$.
With this in mind, we let
\begin{equation}   \label{normdist}
d(\x,L):=\min_{\by \in L} \, \|\bx-\by\| =  \min_{\by \in L} \, \max_{1 \le i \le n} |x_i-y_i|  \,
\end{equation}
denote the minimal distance between $\bx$ and  $L$.
We also let $H(L)$ denote the height of $L$.
In short, $H(L)$ is the volume of the  sub-lattice  $\Z^{n+1} \cap L_0$ where  $L_0$ is  the unique $(d+1)$-dimensional subspace  of $\R^{n+1}$ containing the $d$-dimensional embedding $\{(y_1, \ldots, y_n, 1): \by \in L \}$ of $L$ into $\R^{n+1}$.   This notion of height is relatively standard and is usually referred to as the projective or Weil height of $L$ -- see  \cite{Famous5, MLwd, SchH} for more details.
Note that when $d=0$,  $L$ corresponds to a  rational point $\frac{\bp}{q}:= \big(\frac{p_1}{q}, \ldots \frac{p_n}{q} \big)$ for some $(\bp, q) \in \mathbb{Z}^n \times \mathbb{Z} \setminus \{0\}$.  In turn, we have that
$$
H(L)  \, \asymp \,  \max \{\|\bp\| ,  |q|  \}     \qquad {\rm and }  \qquad d(\x,L)  \, = \,      \max_{1 \le i \le n} \frac{\left|q x_i - p_i  \right|}{|q|}  \, .
$$
Also note that  when $d=n-1$,  $L$ corresponds to a  rational affine hyperplane $\{ \by \in \R^n: \bq \cdot \by = p  \}  $ for some $(\bq,p) \in \mathbb{Z}^n{\setminus}\{0\} \times \mathbb{Z}$.  In turn,  we have that
$$
 H(L)  \, \asymp \,   \max \{ |p|, \|\bq\| \}   \qquad {\rm and }  \qquad d(\x,L)  \, \asymp \,   \frac{ \left| q_1x_1 + \ldots + q_nx_n -p \right| }{\|\bq\|}  \, .
$$
To simplify notation the  symbols $\ll$
and $\gg$ will be used to indicate an inequality with an
unspecified positive multiplicative constant. If $a \ll b$ and $a
\gg b$ we write $ a \asymp b $, and  say that the quantities $a$
and $b$ are \emph{comparable}.  In the above,  the implied `comparability' constants are dependent on $n$.  Thus, up to some multiplicative constants,   the extreme cases $d=0$ and $ d=n-1$  correspond to the standard simultaneous and dual forms of Diophantine approximation.
We  now consider the natural  analogues of the sets $ \Bad_n $ and  $ \Sing_n $  introduced within the framework of simultaneous Diophantine approximation  in  \S\ref{BM}.  With this in mind, we start by stating a Dirichlet type  theorem for approximation by $d$-dimensional rational subspaces.   Throughout, given $ n \in \N$ and  $ d \in \{0,1, \ldots, n-1\} $,  we let
$$
\omega_{d}:=\frac{d+1}{n-d}  \, .
$$

\begin{thm} \label{Dir-d} Let $n \in \N$ and $d$ be integer satisfying $0 \le d \le n-1$.  Then
 for any $\vx \in \RR^n$  there exists a constant $ c=c(n,d,\bx) > 0$,  such that for any  $N \in \mathbb{N}$ there exists a $d$-dimensional rational affine subspace $L \subset \R^n$,  such that
  \begin{equation}\label{slv-d}
d(\vx,L) \, \leq \, c \,  H(L)^{-1}  N^{-\omega_d} \qquad { and }  \qquad H(L)\leq N . \end{equation}
\end{thm}

\noindent The above statement  is a consequence of standard tools from the geometry of numbers such as  Minkowski's  second convex body theorem and Mahler's theory for compound bodies.  For completeness,  the details are given in \S\ref{App} (see Proposition~\ref{propa1} in \S\ref{DTTMA} and Proposition~\ref{lemB3A} in \S\ref{Sec_A2.1}).   In turn, the theorem gives rise to the following statement.

\begin{cor}\label{CorDir-d}
 Let $n \in \N$ and $d$ be integer satisfying $0 \le d \le n-1$.  Then
 for any $\vx \in \RR^n$ with at least $(d+1)$-rationally independent coordinates, there exists a constant $ c=c(n,d,\bx) > 0$  and  infinitely many  $d$-dimensional rational affine subspaces $L \subset \R^n$  such that
  \begin{equation}\label{slv-cd}
d(\vx,L) \, \leq \, c \, H(L)^{-1-\omega_d} \,  . \end{equation}
\end{cor}

\bigskip

\begin{remark}  Proposition~\ref{lemB3A} in \S\ref{Sec_A2.1} together with the remark immediately proceeding it, enables us to explicitly compute the constant $c=c(n,d,\bx)$ appearing in  the above results.   Observe that if we restrict $\bx$ to a bounded subset of $\R^n$, then the constant $c$ can be made to be independent of $\bx$.  In particular, if $\bx \in [0,1]^n$ then in the simultaneous (resp. dual) case we can replace $H(L)$ by $|q|$ (resp. $\|\bq\|$) in the theorem and corollary, and the inequalities corresponding to  \eqref{slv-d}  and \eqref{slv-cd} remain valid if we translate $\bx$ by an integer vector. Thus, up to a constant dependent only on the dimension $n$, Theorem~\ref{Dir-d}  and its corollary coincide with the classical simultaneous and dual forms of Dirichlet theorem and its corollary.
\end{remark}
\medskip

Taking our lead from the classical simultaneous and dual settings, we say that a point
$\bx \in \R^n$ is {\em $d$-singular}  if for any given $\eps \in (0,1) $ and $N > N_0( \bx, \eps, d) $ sufficiently large,  there exists a  $d$-dimensional rational affine subspace $L \subset \R^n$  such that
\begin{equation}\label{equ-gen}
d(\vx,L) \, \leq \, \eps \,  H(L)^{-1}  N^{-\omega_d} \qquad { and }  \qquad H(L)\leq N . \end{equation}
On the other hand, we say that a point
$\bx \in \R^n$ is {\em $d$-badly approximable} if there
exists a constant  $\eps=\eps(\bx) \in (0,1)$ so that
  \begin{equation}\label{slv-bad-d}
  d(\vx,L) \, > \, \eps \,  H(L)^{-1-\omega_d}
  \end{equation}
for all  $d$-dimensional rational affine subspaces $L \subset \R^n$. Finally, we let $\Sing_n^d$ (resp. $\Bad_n^d$)  denote the set of $d$-singular (resp. $d$-badly approximable) points in  $\R^n$.

The following shows that the well known classical equivalence between the simultaneous and dual singular points (and indeed badly approximable points) holds in the general context of approximation by
$d$-dimensional rational affine subspaces. We provide a proof in  \S\ref{sec-reform}.

 \begin{prop}\label{prop-equiv}
 Let $n \in \N$ and $d$ be integer satisfying $0 \le d \le n-1$.  Then,
 $$
 \Sing_n := \Sing_n^0 = \Sing_n^d  \qquad { and} \qquad \Bad_n := \Bad_n^0 = \Bad_n^d \, .
 $$
\end{prop}

\bigskip

\begin{remark} \label{bull} Note that for the purpose of  defining $d$-singular and $d$-badly approximable points it makes no difference whether the minimal distance $  d(\vx,L) $  is defined via the maximum norm (as in \eqref{normdist}) or some other  norm (such as the Euclidean norm).  The point is that these notions are not sensitive to the actual value of the constant $c=c(n,d,\bx)$ appearing in  Theorem~\ref{Dir-d} and its corollary.  Thus, most importantly, the set $\Sing_n^d$ (resp. $\Bad_n^d$)  coincides with the classical simultaneous singular (resp. badly approximable) set when $d=0$ and  the dual singular (resp. badly approximable) set when $d= n-1$. However,  when it comes to defining  the `right' notion of $d$-Dirichlet improvable points it is paramount that $c$ is optimal and that (an appropriate version of) Theorem~\ref{Dir-d} coincides with the classical simultaneous and dual forms of Dirichlet theorem.  Clearly, in its current form it fails to do so. In \ourappendix{} of this paper  we shall propose two versions (algebraic and geometric) of  Theorem~\ref{Dir-d} leading to corresponding notions of Dirichlet improvable points that rectify this issue.
Although it is not particularly relevant within the context of our main result, we hope \ourappendix{} is of independent interest. {In short, we show that within either setting the corresponding  $d$-Dirichlet improvable sets $\DI_n^d$ are all equivalent and coincide with the classical set $\DI_n$; that is, the set of Dirichlet improvable points in $\R^n$ defined via either the classical simultaneous ($d=0$) or dual ($d=n-1$) form of  Dirichlet's theorem (both gives rise to the same set thanks to the aforementioned statement of Davenport $\&$ Schmidt).  It thus follows that the $d$-Dirichlet improvable sets defined via the algebraic and geometric settings also coincide.}
\end{remark}

In order to state our main result, it is convenient to introduce the notion of exponents of Diophantine approximation.

\begin{definition} \label{expo}
Let $d$ be an integer with $0\le d \le n-1$ and let  $ \bx \in \R^n$. We define the \emph{$d^{\,th}$ ordinary exponent} ${\omega}_{d}(\bx)$  (resp. the \emph{$d^{\,th}$ uniform exponent} $\hat{\omega}_{d}(\bx)$) as the supremum of the real numbers $\omega$ for which there exist $d$-dimensional  rational affine subspaces $L \subset \mathbb{R}^{n}$ such that
 $$d(\vx,L) \le H(L)^{-1}  N^{-\omega} \qquad { and }  \qquad H(L)\le N . $$
  for arbitrarily large real numbers $N$  (resp. for every sufficiently large real number $N$).  \\
  \end{definition}

\begin{remark} By definition, whenever  ${\omega}_{d}(\bx)$  is finite,  there exists infinitely many $d$-dimensional rational affine subspaces $L \subset \R^n$  such that
$$d(\vx,L) \le H(L)^{-1-\omega} $$
if $ \omega < {\omega}_{d}(\bx)$,  and if $ \omega > {\omega}_{d}(\bx)$ there are at most finitely many such subspaces $L \subset \R^n$.
\end{remark}

\begin{remark} \label{vsing}
In \cite{VarPrinc} a point $\bx \in \R^n$ satisfying $\hat{\omega}_{0}(\bx)> 1/n$ (equivalently $\hat{\omega}_{n-1}(\bx)> n$) is called very singular and the set of such points is denoted by $\VSing_n$.  In the context of approximation by $d$-dimensional rational affine subspaces, it is natural to define the notion of {\em $d$-very singular points} as points in the set
$$
\VSing_n^d \, := \, \{\bx\in \RR^n : \hat{\omega}_{d}(\bx)  >   \omega_d  \}  \, .
$$
As is the case of badly approximable and singular sets, it turns out  that the sets $\VSing_n^d $  ($0\le d \le n-1$)  are the same -- see Remark~\ref{shproof} below.
By definition,  a $d$-very singular point is $d$-singular and since both notions are independent of $d$ we can simply write
$
\VSing_n  \subseteq \Sing_n \,
$.

\end{remark}

Within the classical simultaneous and dual forms of Diophantine  approximation, the above exponents were introduced by Khintchine \cite{Khin1,Khin2} and  Jarn\'ik \cite{JAR}  in the nineteen twenties and thirties.  For $ n \ge 3$, the  intermediate exponents (i.e.,  those corresponding to $ 1 \le d \le n-2 $)   were formally  introduced  by Laurent \cite{MLwd} in 2009 but  had  implicitly been studied by Schmidt \cite{SchH} some fifty years earlier.   Clearly, for any $ \bx \in \R^n$ we have that ${\omega}_{d}(\bx) \ge \hat{\omega}_{d}(\bx)$  and a direct consequence of  Theorem \ref{Dir-d} is that
$$
{\omega}_{d}(\bx) \ge \hat{\omega}_{d}(\bx) \ge \omega_{d}:=\frac{d+1}{n-d}  \, .
$$
Observe that $\omega_{0}= \frac1n$ and $ \omega_{n-1} = n$.  Thus,  when $d=0$ (resp. $d=n-1$) the quantity  $\omega_{d}$  coincides  with the exponent appearing in the classical simultaneous (resp. dual)  form of Dirichlet's theorem.
Another reasonably straightforward  consequence, this time of the Borel-Cantelli lemma from probability theory,  is that
\begin{equation}\label{svae} {\omega}_{d}(\bx) =  \omega_{d}    \quad   {\rm  for \ almost \ all \ }  \bx \in \R^n \, .\end{equation}

The following elegant transference principle enables us to transfer information
between the ordinary Diophantine exponents ${\omega}_{d}(\bx)$ associated with approximating  points in $ \bx \in \R^n$
by $d$-dimensional rational subspaces of $\R^n$.  It makes sense to include the statement at this point since one of the conditions turns up in the statement of our main theorem.

\begin{thm}[Laurent \& Roy]\label{GUGD}
Let $n\ge2$.   For any point $\bx \in \R^n$ with $1,x_1, \ldots,x_n$ linearly independent over $\Q$,  we have that $ \omega_{0}(\btheta) \ge \omega_{0}$ and
\begin{equation}\label{gugd}  \frac{d \, \omega_{d}(\btheta)}{\omega_{d}(\btheta)+d+1} \,  \le \, \omega_{d-1}(\btheta) \,  \le  \, \frac{(n-d) \, \omega_{d}(\btheta)-1}{n-d+1}  \qquad ( 1\le d \le n-1) \, . \end{equation}
If $\omega_{d}(\bx)=\infty$, the  left hand side in \eqref{gugd} is replace by $d$. Furthermore, given any $\tau_{0}, \ldots, \tau_{n-1} \in [0, \infty]$ with $\tau_{0}\ge \omega_{0}$ and
\begin{equation}\label{gugdsv}    \frac{d \, \tau_{d}}{\tau_{d}+d+1} \le \tau_{d-1} \le \frac{(n-d)\, \tau_{d}-1}{n-d+1}  \qquad ( 1\le d \le n-1) \, ,
\end{equation}
there exists a point $\bx \in \R^n$ with $1,x_1, \ldots,x_n$ linearly independent over $\Q$ such that $ \omega_{d}(\btheta)~=~\tau_{d} $ and $ \hat{\omega}_{d}(\btheta)= \omega_d $ for $0\le d \le n-1$.
\end{thm}

\noindent The transference inequalities \eqref{gugd}  are due to Laurent \cite{MLwd}. Equivalently, they can be re-written in the language of Schmidt \cite{SchH} as the {\em Going-up transfer}
\begin{equation}\label{gugdUP} \omega_{d+1}(\btheta) \,  \ge  \, \frac{(n-d) \, \omega_{d}(\btheta)+1}{n-d-1}  \qquad ( 0\le d \le n-2)
\end{equation}
and the {\em Going-down transfer}
\begin{equation}\label{gugdDOWN}
\omega_{d-1}(\btheta) \,  \ge  \, \frac{d \, \omega_{d}(\btheta)}{\omega_{d}(\btheta)+d+1}   \qquad ( 1\le d \le n-1)  \, .
\end{equation}
As pointed in \cite{MLwd}, on iterating \eqref{gugdUP}  and  \eqref{gugdDOWN}  we obtain Khintchine's classical transference principle \cite{Khin1}:
 $$
 \frac{\omega_{n-1}(\bx)}{(n-1)\omega_{n-1}(\bx)+n} \le \omega_{0}(\bx) \le \frac{\omega_{n-1}(\bx)-n+1}{n}  \, .
$$
 Thus the transference inequalities \eqref{gugd}  of Laurent naturally split  those of Khintchine  relating the simultaneous and dual exponents $\omega_{0}(\bx) $ and $ \omega_{n-1}(\bx)$.  The furthermore part of Theorem~\ref{GUGD},  shows that transference inequalities of Laurent are optimal and was proved by Roy \cite{Royspec}.  It extends  the classical work of Jarn\'ik \cite{JAR} showing that  Khintchine's transference principle is optimal.

%
%
%
%

\begin{remark} \label{shproof}
The  Laurent transference inequalities \eqref{gugd} are equally valid for the uniform exponents. Indeed, Laurent's proof for the ordinary exponents can be naturally adapted to the uniform setting -- see for example \cite{GerActa}.  With \eqref{gugd} for uniform exponents at hand, it is easily seen that for any $\bx \in \R^n$ and $1\le d \le n$, the statement that  $\hat{\omega}_d(\bx)=\omega_d$ is equivalent to $\hat{\omega}_{d-1}(\bx)=\omega_{d-1}$. Hence, it follows that the very singular  sets  $\VSing_n^d \, (0\le d \le n-1)$  discussed within  Remark \ref{vsing} are equivalent.
\end{remark}

\bigskip

As usual let $d$ be an integer with $0\le d \le n-1$. Then given a real number  $\tau \ge 0 $,  consider the Diophantine sets
$$\bfW_n^d( \tau )\, := \, \{\bx\in \RR^n : {\omega}_{d}(\bx)\ge  \tau  \}$$
and
$$\bfE_n^d( \tau )\, := \, \{\bx\in \RR^n : {\omega}_{d}(\bx)  =  \tau  \}   \, . $$
In dimension one, the latter corresponds to the exact order sets first studied by G\"uting within the context of Mahler's classification of transcendental numbers  -- see \cite{BDV-gut} and references within for further details.  Note that by definition, for any $0\le d \le n-1$ we have that
$$
\bfW_n^d( \tau  ) \, = \, \R^n  \qquad   {\rm if } \quad \tau  \leq \omega_d
$$
and
$$
\Bad_n\cap \bfW_n^{d}({\tau })=\emptyset \qquad   {\rm if } \quad \tau  > \omega_d   \, .
$$

\noindent Note that in view of \eqref{svae}, the set $\bfE_n^d( \omega_d )$ is of full $n$-dimensional Lebesgue measure and since $ \bfE_n^d( \tau )  \subseteq \bfW_n^d( \tau )$, it follows that $\Bad_n\cap \bfE_n^{d}({\tau })=\emptyset $ if  $\tau  > \omega_d  $.

Using the parametric geometry of numbers a la Schmidt $\&$ Summerer \cite{SSfirst,SS} and Roy \cite{Roy}, we prove  a stronger version of Theorem~\ref{thslv1} that involves the exact order sets $\bfE_n^d( \tau )$ and the following quantitative form of the set of Dirichlet improvable points $\DI_n$. Given $\eps\in(0,1)$, let
$ \DI_n(\eps) $  denote the set of $\bx \in \R^n$  such that, for all $N > N_0( \bx, \eps) $ sufficiently large there exists $\bq \in \mathbb{Z}^n{\setminus}\{0\}$ such that \eqref{equ-basic} holds.

\begin{remark} \label{dualeps}
By definition, it follows that
$$
\DI_n=\bigcup_{\eps\in(0,1)}\DI_n(\eps)  \, .
$$
Recall, that in view of Davenport $\&$ Schmidt \cite[Theorem~2]{DavSch} we can define $ \DI_n $ via either the `simultaneous' inequality \eqref{DIsv1} or the `dual' inequality  \eqref{equ-basic}  -- both give rise to the same set. However, if we fix $\eps > 0$ and  a point $ x \in \DI_n(\eps) $, then it is not necessarily true that for all $N$ sufficiently large there exists $q \in \mathbb{Z}$ such that \eqref{DIsv1} holds   (for the same $\eps$). In view of this, we emphasise the fact  that when referring to the quantitative  Dirichlet improvable set $\DI_n(\eps)$ it will always be via the `dual' inequality  \eqref{equ-basic}.
\end{remark}

The following theorem constitutes our main result.

\begin{thm}\label{thm-main}
  Let $n\ge 2$ and $\eps\in (0,1)$. Then, given any  $n$-tuple of real numbers  $\tau_{0}, \ldots, \tau_{n-1} \in [0, \infty]$ with $\tau_{0}\ge \omega_{0}$ and $\tau_{d}$ $( 1\le d \le n-1) $ satisfying  \eqref{gugdsv}, the set
  \[\left(\bigcap_{d=0}^{n-1} \bfE_n^d({\tau}_{d}) \cap \Big(\DI_n(\eps) \setminus \DI_n(\eps \e^{-10(n+1)^2(n+10)})\Big)\right)\setminus  (\Bad_n\cup \Sing_n)\] {has continuum many points}.  In particular, for any $\tau \geq {\omega}_{d}\,$,  the set $(\DI_n\cap \bfE_n^d(\tau))\setminus (\Bad_n\cup \Sing_n)$ {has continuum many points}.
\end{thm}

\noindent
Note that on taking $\tau =  {\omega}_{d}\,$  in the `in particular' part of the Theorem~\ref{thm-main}, we immediately  obtain the statement of  Theorem~\ref{thslv1}.

\section{Preliminaries}

In this section we start by recalling aspects of the theory of  parametric geometry of numbers that will be used in establishing Theorem~\ref{thm-main}. We then use this to essentially  reformulate the Diophantine sets appearing in the statement of Theorem~\ref{thm-main} in terms of successive minima.   Moreover, we will see that the proof of Proposition~\ref{prop-equiv} is a pretty straightforward application of this reformulation.

\subsection{The parametric geometry of numbers}\label{PGN}

Fix $n \in \N$ and $\bx\in \RR^n$.  For each real number $t\ge 0$, consider the convex body
\begin{equation}\label{def-sc}
  \sC_{\bx}(\e^t):= \left\{ \by \in \RR^{n+1}: |y_i|\le 1  \ (1\le i\le n), \ \Big|\sum_{i=1}^{n}y_ix_i+y_{n+1}\Big|\le e^{-t}\right\}.
\end{equation}
Then, for each $i=1, \ldots, n+1$ and $t>0$, let
\begin{equation}\label{vb8D}
\lambda_{\bx,i}(t):=\lambda_i\big(\ZZ^{n+1}, \sC_{\bx}(\e^t)\big)
\end{equation}
denotes the $i$-th successive minima of the convex body $\sC_{\bx}(\e^t)$ with respect to the lattice $\ZZ^{n+1}$. In other words,  $\lambda_i\big(\ZZ^{n+1}, \sC_{\bx}(\e^t)\big)$ is the smallest real number $\lambda$ such that the rescaled convex body $\lambda \, \sC_{\bx}(\e^t)$ contains at least $i$ linearly independent points of $\Z^{n+1}$. In turn, following  Schmidt $\&$ Summerer \cite{SS},  we let
 \begin{equation}\label{def-convexbodySV}
 L_{\bx, i}(t): = \log \lambda_{\bx,i}(t)=\log\lambda_i\big(\ZZ^{n+1}, \sC_{\btheta}(\e^t)\big)  \qquad (t\ge 0, \ 1 \le i \le n+1)
\end{equation}
 and consider the map
\begin{equation}\label{def-convexbody}
\bL_\bx: [0, \infty)\rightarrow \RR^{n+1} \ : \ t \  \rightarrow \bL_{\bx}(t):=\big(L_{\bx,1}(t),\ldots, L_{\bx, n+1}(t)\big) \, .
\end{equation}

The following  notion was introduced by Roy in \cite[Defintion 4.5]{Royspec}. It generalises the $(n+1)$-systems of Schmidt $\&$ Summerer \cite{SS}.  In short, these `systems' incorporate desirable behavior of the maps $\bL_{\bx}$ that in turn lead to desirable approximation results.

\begin{definition} \label{Roybaby}
 Let $I$ be an subinterval of $[0,\infty)$ with non-empty interior. A \emph{Roy $(n+1)$--system} on $I$ is a continuous piecewise linear map $\bP=(P_1,\ldots, P_{n+1}): I \rightarrow \RR^{n+1}$ with the following properties:
 \begin{itemize}
   \item For each $t\in I$, we have $0\le P_1(t)\le \cdots \le P_{n+1}(t)$ and $P_1(t)+\cdots+ P_{n+1}(t)=t$.
   \item If $I'\subset I$ is a nonempty open subinterval on which $\bP$ is differentiable, then there are integers   $r_1,r_2$ with $1\le r_1\le r_2\le n+1$ such that the functions $P_{r_1}, P_{r_1+1},\ldots, P_{r_2}$ coincide on the whole interval $I'$ and have slope $1/(r_2-r_1+1)$ on $I'$, while all other components $P_i$ of $ \bP$ have slope $0$ on $I'$.
   \item If $t$ is an interior point of $I$ at which $\bP$ is not differentiable and if $r_1, r_2$, $s_1, s_2$ are integers for which
       \begin{equation*}
         P_i'(t^-)=\frac{1}{r_2-r_1+1}\quad (r_1\le i\le r_2) \quad \text{ and } \quad P_i'(t^+)=\frac{1}{s_2-s_1+1}\quad (s_1\le i\le s_2) \, ,
       \end{equation*}
         and if $r_1\le s_2$, then we have that $P_{r_1}(t)=P_{r_1+1}(t)=\cdots= P_{s_2}(t)$.
 \end{itemize}
\end{definition}

\noindent Note that, for any piecewise linear function $F:\RR\rightarrow \RR$, the left derivative $F'(t^-)$ and the right derivative  $F'(t^+)$ always exist and the points at which $F$ is not differentiable are just the points with different left and right  derivatives.

\bigskip

\begin{remark} \label{AM}
The $(n+1)$--systems of Schmidt $\&$ Summerer correspond to taking
$ r_1=r_2$ and $s_1=s_2$ in Definition~\ref{Roybaby}.
\end{remark}

A Roy $(n+1)$--system  has the following useful approximation property.  It essentially represents an amalgamation  of
\cite[Theorems 1.3 $\&$ 1.8]{Roy} and \cite[Corollary 4.7]{Royspec} adapted for our purposes.

\begin{thm}\label{thm-roy}
Let $n\in \NN$ and $t_0 \ge 0$. For each $\btheta\in\RR^n$, there exists a Roy $(n+1)$--system $\bP: [t_0, \infty) \rightarrow \R^{n+1}$  such that the function $\bL_{\btheta}-\bP$ is bounded on $[t_0, \infty)$.  Conversely, for each Roy $(n+1)$--system $\bP: [t_0,\infty)\rightarrow \RR^{n+1}$, there exists   $\btheta\in \RR^{n}$ such that  the function $\bL_{\btheta}-\bP$ is bounded on $[t_0, \infty)$. In particular, for each $t\ge t_0$
\begin{equation}\label{vb+}
\|\bL_{\btheta}(t) - \bP(t)\|  \, \le  \,   5 (n+1)^2(n+10)  \, .
\end{equation}
\end{thm}

\bigskip

\begin{proof}[Proof of Theorem \ref{thm-roy}.]
 As already mentioned, Theorem \ref{thm-roy} draws upon the works \cite{Roy, Royspec} of Roy and it is important to  note  that there, the convex body is defined slightly differently  from the one given by  \eqref{def-sc}.  Indeed, for a fixed $\bu\in \R^{n+1} {\setminus}\{0\}$ and each real number $t \ge 0$, Roy works with the convex body
\[\widetilde{\cC}_{\bu}(\e^t):=\left\{\by \in \RR^{n+1}: \|\by\|\le 1, \quad |\by\cdot \bu|\le \e^{-t}\right\} \, . \]
Now for any fixed $\bx \in \R^n$, let $\bx':=(\bx,1) \in \R^{n+1}$  and so by definition
\[\widetilde{\cC}_{\bx'}(\e^t)=\left\{\by \in \RR^{n+1}: \|\by\|\le 1, \quad \Big|\sum_{i=1}^{n}y_ix_i+y_{n+1}\Big|\le \e^{-t}\right\}.\]
Furthermore,  let $\tilde{\bL}_{\bx'}$ denote the function corresponding to \eqref{def-convexbody} with  $\sC_{\btheta}(\e^t)$ replaced by $\widetilde{\cC}_{\bx'}(\e^t)$ within \eqref{def-convexbodySV}.  It is not difficult to see that our  convex body and the associated map, which are convenient for what we have in mind,  are closely related to those of Roy and indeed Schmidt $ \&$ Summerer:  for any fixed $\bx\in \R^n$ and any $ t \ge 0$
\begin{equation*}\label{compare-convexbody}
  \widetilde{\cC}_{\bx'}(\e^t) \ \subset \ \sC_{\bx}(\e^t)\ \subset \ (n\|\bx\|+1)\, \widetilde{\cC}_{\bx'}(\e^t)
\end{equation*}
and thus it follows that
\begin{equation}\label{compare-l-function}
\|\bL_{\bx}(t)-\tilde{\bL}_{\bx'}(t)\| \le \log (n\|\bx\|+1) \, .
\end{equation}

We now proceed with establishing the theorem.  On combining \cite[Theorem1.3]{Roy} and \cite[Lemma 2.10]{Roy},   we find that for any $\bx\in \RR^n$ there exists a $(n+1)$--system (see Remark \ref{AM}) $\widetilde{\bP}$ such that the function $\tilde{\bL}_{\bu_{\bx'}}-\widetilde{\bP}$ is bounded for all $t\ge t_0$. Here $\bu_{\bx'}$ denotes the unit vector associated to $\bx' \in \R^{n+1}$.  Note that  $\tilde{\bL}_{\bu_{\bx'}}$ and $\tilde{\bL}_{\bx'}$ only differ by a constant and so the first part of the theorem follows on using \eqref{compare-l-function} and the fact that by definition any  $(n+1)$--system is a  Roy $(n+1)$--system.
  Regarding the converse part,  it follows via \cite[Corollary 4.7]{Royspec} that for any given  Roy $(n+1)$--system $\bP: [t_0,\infty)\rightarrow \RR^{n+1}$ and any $\eps>0$,  there is a  $(n+1)$--system $\widetilde{\bP}$ such that
\begin{equation}\label{compare-roy-normal}
  \|\bP(t)-\widetilde{\bP}(t)\|  \le \eps   \quad \text{ for all } t\ge t_0.
\end{equation}
In view of \cite[Theorem 8.1]{Roy}, there exists a unit vector $\bu \in \RR^{n+1}$ such that
\begin{equation}\label{equ-4-1}
 \|\widetilde{\bP}(t)-\tilde{\bL}_{\bu}(t)\|  \le 3(n+1)^2(n+10) \quad \text{ for all } t\ge t_0.
\end{equation}
Without loss of generality, we can assume that $|u_{n+1}|=\|\bu\|:=\max \{|u_1|, \ldots, |u_{n+1}|\}$ and so  $(n+1)^{-1/2}\le |u_{n+1}|\le 1$. Now  let $$\btheta:=(u_1u_{n+1}^{-1},\ldots, u_nu_{n+1}^{-1})\in \RR^n \, . $$ Then,  $\|\bx\|\le 1$ and
\begin{equation}\label{equ-4-2}
  \|\tilde{\bL}_{\bx'}(t)-\tilde{\bL}_{\bu}(t)\| \le \frac12 \log(n+1)\quad \text{ for all } t\ge t_0.
\end{equation}
 The upshot is that on using \eqref{compare-roy-normal} with  $\eps:=\log(n+1)/2$, \eqref{compare-l-function}, \eqref{equ-4-2} and \eqref{equ-4-1} in that order,  we obtain that
\begin{eqnarray}\label{equ-3-2}
  \|\bL_{\bx}(t)-\bP(t)\| & \le   &  \frac12 \log(n+1) +  \log(n+1)+ \frac12 \log(n+1)  + 3(n+1)^2(n+10)   \nonumber \\[2ex]
  & < &  5(n+1)^2(n+10)\quad \text{ for all } t\ge t_0.
\end{eqnarray}
This completes the proof of  Theorem \ref{thm-roy}.
\end{proof}




The following notion of non-equivalent systems will prove to be useful.

\begin{definition} \label{neRS}
Two Roy $(n+1)$--systems $\bP_1$  and $\bP_2$ defined on the same
subinterval $I$ of $[0,\infty)$ are said to be \emph{non-equivalent} if
there exists some $t\in I$ such that $$ |\bP_1(t)-\bP_2(t)|>10
(n+1)^2(n+10)  \, . $$
\end{definition}

\noindent By definition, it follows that no point in $\RR^{n}$ can be
close (in the sense of Theorem \ref{thm-roy}) to  two non-equivalent Roy
$(n+1)$--systems defined on $[t_0, \infty)$ at the same time.

\subsection{Expressing Diophantine sets via successive minima}\label{sec-reform}

We give a reformulation of the Diophantine sets associated with Theorem \ref{thm-main}  in   terms of the function $\bL_{\bx}$. This is at the heart of its proof -- it  brings into play the parametric geometry of numbers.  Also,  we shall see that the equivalence of the $d$-badly approximable sets $\Bad_n^d$ (resp. the $d$-singular sets $\Sing_n^d$) is in essence a direct consequence of the reformulation. Indeed, we start with this in mind.

Let $n \ge 2 $ and $ 0 \le d \le n-1$. It can be verified, by using the lemma appearing in \cite[Section 4]{BugLau2} and appropriately adapting the proof of the proposition  in  \cite[Section 4]{BugLau2},   that

\begin{itemize}
  \item $\bx \in \Bad_n^d$ if and only if there exists a constant  $\delta > 0$ such that for all sufficiently large  $t$
\begin{equation}\label{equ-equi}
  \frac{(n-d)t}{n+1} - (L_{\bx,1}(t) + \cdots + L_{\bx,n-d}(t)) \le \delta.
\end{equation}
  \item $\bx \in \Sing_n^d$ if and only if for any  $\delta > 0$ there exists a constant $t_0= t_0({\delta})  > 0 $ such that for all $t \ge  t_0$
\begin{equation}\label{equ-equiSING}
  \frac{(n-d)t}{n+1} - (L_{\bx,1}(t) + \cdots + L_{\bx,n-d}(t)) \ge \delta.
\end{equation}
\end{itemize}
For the sake of completeness, in \S\ref{App},  we will provide the details of how these  equivalences follow from \cite[Section 4]{BugLau2}.    We  can now swiftly  show that
$$
 \Bad_n^d   =   \Bad_n^{n-1}    \quad  {\rm and } \quad     \Sing_n^d   =   \Sing_n^{n-1}    \qquad (0 \le d \le n-2)  \, ;
$$
that is to say that any $d$--badly approximable set (resp. $d$--singular set) is equivalent to the dual set.
This will of course  establish Proposition \ref{prop-equiv}.

 \medskip

 \begin{proof}[Proof of Proposition \ref{prop-equiv}]

For simplicity, given $\bx \in \R^n$ we let
\[g_{\bx,i}(t):= \frac{t}{n+1} - L_{\bx,i}(t)    \qquad (0 \le i \le n+1)  \,  .\]
By definition the quantity $L_{\bx,i}$ is increasing with $i$ and so it follows that
\begin{equation}\label{e-decreasing}
  g_{\bx,1}(t) \ge g_{\bx, 2} (t) \ge \cdots \ge g_{\bx,n+1}(t).
\end{equation}
In view of  Minkowski's  second convex body theorem, for any $\bx \in \R^n$ we have that
 $$L_{\btheta,1}(t) + L_{\bx,2}(t)+ \cdots + L_{\btheta,n+1}(t) = t +O(1) \, .  $$
 Thus, there exists a positive constant $c = c(n) > 0$ depending only on $n$  such that
 \begin{equation}\label{e-min-sec}
   g_{\bx,1}(t)+g_{\bx,2}(t)+\cdots+g_{\bx,n+1}(t)   \ge - c \, .
 \end{equation}

 Now suppose $\bx \in \Bad_n^d$. Then in view of
\eqref{equ-equi} and \eqref{e-min-sec}, it follows that
 \[\sum_{i=n-d+1}^{ n+1} g_{\bx,i}(t) \ge -\delta - c \,  \]
 which together with \eqref{e-decreasing} implies that
 \[g_{\bx,i}(t)\ge \frac{1}{d+1}\sum_{j=n-d+1}^{d+1} g_{\bx,j}(t) \ge -\frac{\delta + c }{d+1} \qquad   ( 1\le i\le n-d) \, .\]
 In turn, on using \eqref{e-decreasing} again, we find that
 \begin{eqnarray*}
 g_{\bx,1}(t) \le \sum_{i=1}^{n-d} g_{\bx,i}(t)- \sum_{i=2}^{n-d} g_{\bx,i}(t)& \le  &  \delta + (n-d-1)\frac{\delta + c }{d+1} \\ & < & \frac{n}{d+1}(\delta + c) .
\end{eqnarray*}
In other words, \eqref{equ-equi} holds with $d=n-1$ and so $\bx \in \Bad_n^{n-1}$. For the converse, simply observe that if \eqref{equ-equi} holds with $d=n-1$ then for any other $ 0 \le d \le n-2$
$$
\sum_{i=1}^{n-d} g_{\bx,i}(t) \stackrel{\eqref{e-decreasing}}{\le} (n-d)  \; g_{\bx,1}(t) \le  (n-d) \delta  \, .
$$
In other words, $\bx \in \Bad_n^{d}$ and this thereby completes the proof of the badly approximable part of the proposition. The  proof of the singular part is similar with the most obvious modifications (namely, using  \eqref{equ-equiSING} instead of \eqref{equ-equi}) and will be left for the reader.

\end{proof}

\begin{remark} \label{ohyaohya}
  In \S\ref{App}, apart from providing details of the statements associated with \eqref{equ-equi} and \eqref{equ-equiSING}, we  give a `dynamical' proof of Proposition~\ref{prop-equiv}.  In addition to  providing an alternative insight,  it has the advantage of being self-contained in that it avoids appealing to \eqref{equ-equi} and \eqref{equ-equiSING} which rely on the lemma  and the arguments  appearing in  \cite[Section 4]{BugLau2}.
\end{remark}

The following statement summarises the above findings concerning the  badly approximable and singular sets and deals with the other remaining  Diophantine sets associated with Theorem \ref{thm-main}.

\begin{lemma}\label{DS}
Let $\btheta\in\RR^n$. Then
\begin{enumerate}
\item
{$\btheta \in \DI_n(\eps)$} if and only if for all sufficiently large  $t$
\begin{equation}\label{DS+}
\frac{t}{n+1}-L_{\bx, 1}(t) \ge -\frac{\log \eps}{n+1}.
\end{equation}
 \item{$\btheta \in \Bad_n$} if and only if there exists  $\delta>0$ such that
\[\limsup_{t \to \infty} \left(\frac{t}{n+1}-L_{\bx, 1}(t) \right)\le \delta.\]
\item{$\btheta \in \Sing_n$} if and only if for any $\delta>0$ \[\liminf_{t \to \infty} \left(\frac{t}{n+1} - L_{\btheta,1}(t) \right) \ge \delta.\]
\item{$\btheta \in \bfW_n^d({\tau}) $} if and only if \[\liminf_{t \to \infty}\frac{L_{\btheta,1}(t) + \cdots + L_{\btheta,n-d}(t)}{t} \le \frac{1}{1+\tau}  \qquad (0\le d \le n-1).\]
\item{$\btheta \in \bfE_n^d({\tau})$} if and only if \[\liminf_{t \to \infty} \frac{L_{\btheta,1}(t) + \cdots + L_{\btheta,n-d}(t)}{t} = \frac{1}{1+\tau}   \qquad (0\le d \le n-1).\]
\end{enumerate}
\end{lemma}

\begin{proof}[Proof of Lemma \ref{DS}]
Parts 4) and 5)  are a direct consequence of  \cite[Proposition~3.1]{Royspec}. The proof of parts 2) and  3) are a direct consequence of  \eqref{equ-equi} and \eqref{equ-equiSING} respectively  together with Proposition \ref{prop-equiv}. It remains to prove  part 1).  Thus, let $\btheta \in \DI_n(\eps)$ for some  $\eps \in (0,1)$. Then by definition, for all  sufficiently large $t'$
  \begin{equation*}
    |\btheta\cdot\bq-p|\le  \eps \e^{-nt'}  \qquad {\rm  and }  \qquad   \|\bq\|  \le e^{t'}
  \end{equation*}
  always has a solution $(p,\bq)\in \ZZ \times (\ZZ^n \setminus \{\bzero\})$.  This is equivalent to saying that
  for all  sufficiently large $t'$
  \[\lambda_1(\ZZ^{n+1}, \sC_{\bx}(e^{t}))\le e^{t'}, \quad \text{ where } t=(n+1)t'-\log \eps.\]
  The latter is equivalent to
  \[ \frac{t}{n+1}-L_{\bx, 1}(t) \ge -\frac{\log \eps}{n+1}\]
for all sufficiently  $t$, as desired.
\end{proof}

\begin{remark}
It is relatively straightforward to see that the proof of part 1) given above can be easily adapted to establish  \eqref{equ-equi} and \eqref{equ-equiSING} when $d=n-1$.
\end{remark}

\begin{remark} \label{kkk}
For the sake of completeness, it worth mentioning that in \S\ref{DTTGN} we formulate the notion of $d$-Dirichlet improvable sets via successive minima.  The approach taken is in line with that of this section in which the $d$-badly approximable and $d$-singular sets are expressed via \eqref{equ-equi} and \eqref{equ-equiSING}.
\end{remark}



\section{Proof of Theorem \ref{thm-main}}

 Let $ n \ge 2$, $\eps\in (0, 1)$   and $\tau_{0}, \ldots ,\tau_{n-1}$ be as in Theorem \ref{thm-main} and let
 \begin{equation} \label{yesno} \gamma: =- \frac{\log \eps}{n+1} +C_n     \qquad  {\rm where  \  \ } \quad  C_n:=5(n+1)^2(n+10)  \, .
 \end{equation}
 Thus, $C_n$ is simply  the right hand side of inequality  appearing in Theorem~\ref{thm-roy}.   Then, on making use of Theorem \ref{thm-roy} and Lemma~\ref{DS}, it is easily verified  that the proof of Theorem~\ref{thm-main} is reduced to  constructing appropriate  Roy $(n+1)$--systems given by the following statement.



\begin{lemma}\label{PropSyst}
There exists {continuum} many mutually non-equivalent Roy $(n+1)$--systems $\bP: [0,\infty)\rightarrow \RR^{n+1}$, such that
\begin{equation}\label{liminf}
  \liminf_{t \to \infty}  \left( \frac{t}{n+1} - P_1(t)\right) =  \gamma  \ ,
\end{equation}

\begin{equation}\label{limsup}
  \limsup_{t \to \infty}  \left( \frac{t}{n+1} - P_1(t)\right) =  \infty   \ ,
\end{equation}
and

\begin{equation}\label{d-exponents}
 \liminf_{t \to \infty}  \frac{P_1(t)+ \cdots + P_{d}(t)}{t} = \frac{1}{1+ \tau_{n-d}}    \qquad (1 \le d \le n) \, .
\end{equation}
\end{lemma}

\begin{proof}[Proof of Theorem \ref{thm-main} modulo Lemma \ref{PropSyst}.]
Let us assume Lemma~\ref{PropSyst} and let $\bP: [0,\infty)\rightarrow \RR^{n+1}$ be a Roy $(n+1)$--system coming from the lemma.   In view of the converse part of Theorem \ref{thm-roy}, there exists a point  $\bx\in \RR^n$ such that
  $$\|\bL_{\bx}(t) - \bP(t)\| \le  C_n  \,  .  $$

\noindent Then this together with Lemma \ref{DS} and
\begin{itemize}
 \item \eqref{liminf} implies $\bx\in \DI_n(\eps)\setminus \DI_n(\eps e^{-2C_n})$ and $\bx\notin \Sing_n\!$  ,
  \item \eqref{limsup} implies $\bx\notin \Bad_n \!$ ,
  \item \eqref{d-exponents} implies $\bx\in \cap_{d=0}^{n-1} \bfE_n^d({\tau}_{d})$  .
\end{itemize}
The upshot of this is that
\[ \bx \in \left(\bigcap_{d=0}^{n-1} \bfE_n^d({\tau}_{d}) \cap \Big(\DI_n(\eps) \setminus \DI_n(\eps \e^{-2C_n})\Big)\right)\setminus  (\Bad_n\cup \Sing_n)   \, . \]
Furthermore, Lemma~\ref{PropSyst} implies the existence of {continuum} many such Roy $(n+1)$--systems that are mutually non-equivalent.  Thus, in view of the latter (see Definition \ref{neRS})  each such system gives rise to a different point $\bx \in  \R^n$ and this completes the proof of Theorem~\ref{thm-main}.
\end{proof}

\medskip

The proof of Lemma \ref{PropSyst} will occupy the rest of this section.  It will comprise of  three steps. We start by  constructing Roy $(n+1)-$systems on certain finite intervals which will serve as building blocks for the construction of the desired systems associated with Lemma~\ref{PropSyst}.

\subsection{Building Blocks.}\label{Blocks}

 Let $[T_-, T_+]$ be a subinterval of $[0,\infty)$ with non-empty interior and let $(a_-^j)_{1\le j\le n+1}, (a_+^j)_{1\le j\le n+1}$ be sequences of increasing positive numbers satisfying:
\begin{eqnarray}
  \quad a_+^1T_+ &=& a_-^{n+1}T_- -(n+1) \gamma  \label{def-t}\\[2ex]
  \sum_{1\le j\le n+1} a_*^j &=& 1  \label{def-sum}\\[2ex]
    (a_*^{j+1}-a_*^j)T_* &\ge& 4n^2\gamma       \qquad  \forall \quad  1\le j\le n \label{def-dist} \, ,
\end{eqnarray}
where $\gamma$ is as in  \eqref{yesno} and throughout
$$
* := - \ \rm{or} \ + \, .
$$We now construct a Roy $(n+1)-$system
$$\bP=(P_1,\ldots, P_{n+1}): [T_-, T_+]\rightarrow \RR^{n+1} $$
on $[T_-, T_+]$ associated with the sequences  $(a_-^j)$ and  $(a_+^j)$. With this in mind, let
\begin{eqnarray*}
  R_d &:=& \left(a_-^{n+1} + \cdots + a_-^{d+1} +da_-^{d}\right)T_-  \qquad  \forall \quad 1\le d \le n \\[2ex]
  R_{n+1} &:=& (n+1)a_-^{n+1}T_-- n(n+1) \gamma\\[2ex]
  R_{n+2} &:=& (n+1)a_-^{n+1}T_--(n+1)\gamma\\[2ex]
  S_0 &:=& (n+1)a_+^1 T_{+} +(n^2+n)\gamma \\[2ex]
  S_d &:=& \left(a_+^1 + \cdots + a_+^d +(n+1-d)a_+^{d+1}\right)T_+  \qquad  \forall \quad 1\le d \le n.
\end{eqnarray*}

 \noindent In view of \eqref{def-t}, it is easily seen  that  $S_0=R_{n+2}$.  Also, \eqref{def-sum} ensures that $  T_-=R_1  $  and $S_n=T_+$   while      \eqref{def-dist} gives that $R_{n+1}\ge R_n$ and $S_1\ge S_0$. Since  $(a_*^j)$ is strictly increasing, it thus follows that
\[T_-=R_1<R_2<\cdots<R_{n+2}=S_0<S_1<\cdots<S_n=T_+.\]
\\
 Now set
\[P_j(T_-):=a_-^j T_- \qquad  \forall \quad  1\le j\le n+1   \, .   \]

\noindent For $1\le d \le n-1$, on the interval $[R_{d}, R_{d+1}]$,  let the $d$ components $P_1, \ldots , P_{d}$ coincide  and have slope $1/d$ while the components $P_{d+1}, \ldots , P_{n+1}$ have slope $0$  and
\begin{eqnarray*}
P_j(R_{d+1}):=\begin{cases}
                 a_-^{d+1}T_- & \mbox{if }   \ 1 \le j\le d \\[1ex]
                 a_-^{j} T_-   & \mbox{if } \ d+1 \le j \le n+1 \, .
               \end{cases}
\end{eqnarray*}
On the interval $[R_{n},R_{n+1}]$, let the $n$ components $P_1, \ldots , P_{n}$ coincide and have slope $1/n$  while the component $ P_{n+1}$ has slope $0$ and
\begin{eqnarray*}
P_j(R_{n+1}):=\begin{cases}
                 a_+^{1}T_+ & \mbox{if }   \ 1 \le j\le n \\[1ex]
                 a_-^{n+1} T_-   & \mbox{if } \ j =  n+1 \, .
               \end{cases}
\end{eqnarray*}
On the interval $[R_{n+1},R_{n+2}]$, let the $n-1$ components $P_2, \ldots , P_{n}$ coincide and have slope $1/(n-1)$ while the components  $P_1, P_{n+1}$ have slope $0$ and
\begin{eqnarray*}
P_j(R_{n+2}):=\begin{cases}
                 a_+^{1}T_+ & \mbox{if }   \ j = 1 \\[1ex]
                 a_-^{n+1} T_-   & \mbox{if } \  2 \le j\le n+1
                 \, .
               \end{cases}
\end{eqnarray*}
On the interval $[S_0, S_{1}]$, let the $n$ components $P_2, \ldots , P_{n+1}$ coincide and have slope $1/n$ while the component  $P_1$ has slope $0$ and
\begin{eqnarray*}
P_j(S_1):=\begin{cases}
                 a_+^{1}T_+ & \mbox{if }   \ j = 1 \\[1ex]
                 a_+^{2}T_{+}   & \mbox{if } \  2 \le j\le n+1
                 \, .
               \end{cases}
\end{eqnarray*}
Finally, for $1\le d \le n-1$, on the interval $[S_{d},S_{d+1}]$, let the $n-d$ components $P_{d+2}, \ldots, P_{n+1}$ coincide and have slope $1/(n-d)$ while
the components $P_1 \ldots , P_{d+1}$ have slope $0$  and
\begin{eqnarray*}
P_j(S_{d+1}):=\begin{cases}
                 a_+^{j}T_{+} & \mbox{if }   \ 1 \le j \le d+1  \\[1ex]
                 a_+^{d+2}T_{+} & \mbox{if }   \ d+2 \le j \le n+1
                 \, .
               \end{cases}
\end{eqnarray*}
In particular, since $ S_n = T_+$ it follows that
\[P_j(T_+)=a_+^j T_+   \qquad  \forall \quad  1\le j\le n+1   \, . \]
\\

\noindent Figure \ref{figA} below  represents the combined graph of the functions $P_1,\ldots, P_{n+1}$ over the interval $[T_-,T_{+}]$.  Note that if we set $\gamma=0$ and $a_{+}^j=a_{-}^j$ for $j=1, \ldots , n+1$, our construction reduces to Roy's construction in \cite[Section~5]{Royspec} -- in particular, see  \cite[Figure~5]{Royspec} .

\begin{figure}[!h]
 \begin{center}
 \begin{tikzpicture}[scale=0.3]

\draw[black, semithick] (2,4)--(0,0) node [left,black] {$a_-^1T_-$};
\draw[black, semithick] (2,4)--(0,4) node [left,black] {$a_-^2T_-$};
\draw[black, semithick] (5,7)--(0,7) node [left,black] {$a_-^{n-1}T_-$};
\draw[black, semithick] (9,9)--(5,7);
\draw[black, semithick] (9,9)--(0,9) node [left,black] {$a_-^nT_-$};
\draw[black, semithick] (9,9)--(12,10);
\draw[black, semithick] (12,10)--(37,10) node [right,black] {$a_+^{1}T_+$};
\draw[black, semithick] (12,10)--(14,11);
\draw[black, semithick] (14,11)--(0,11) node [left,black] {$a_-^{n+1}T_-$};
\draw[black, semithick] (23,14)--(14,11);
\draw[black, semithick] (23,14)--(37,14) node [right,black] {$a_+^{2}T_+$};
\draw[black, semithick] (31,18)--(37,18) node [right,black] {$a_+^{n-1}T_+$};
\draw[black, semithick] (31,18)--(35,22);
\draw[black, semithick] (35,22)--(37,26) node [right,black] {$a_+^{n+1}T_+$};
\draw[black, semithick] (35,22)--(37,22) node [right,black] {$a_+^{n}T_+$};

\draw[black, semithick] (23,14)--(25,15);
\draw[black, dotted, semithick] (25,15)--(26,15.5);
\draw[black, dotted, semithick] (29,17)--(28,16.5);
\draw[black, semithick] (29,17)--(31,18);

\draw[black, semithick] (2,4)--(2.5,4.5);
\draw[black, dotted, semithick] (3,5)--(2.5,4.5);
\draw[black, dotted, semithick] (4,6)--(4.5,6.5);
\draw[black, semithick] (5,7)--(4.5,6.5);

\draw[black, dotted] (0,11)--(0,-4) node [below,black] {$T_-=R_1$};
\draw[black, dotted] (2,11)--(2,-2) node [below,black] {$R_2$};
\draw[black, dotted] (5,11)--(5,-2) node [below,black] {$R_{n-1}$};
\draw[black, dotted] (9,11)--(9,-2) node [below,black] {$R_{n}$};
\draw[black, dotted] (12,11)--(12,-4) node [below,black] {$R_{n+1}$};
\draw[black, dotted] (14,11)--(14,-2) node [below,black] {$R_{n+2}$};
\draw[black, dotted] (23,14)--(23,-2) node [below,black] {$S_{1}$};
\draw[black, dotted] (31,18)--(31,-2) node [below,black] {$S_{n-2}$};
\draw[black, dotted] (35,22)--(35,-2) node [below,black] {$S_{n-1}$};
\draw[black, dotted] (37,26)--(37,-4) node [below,black] {$T_+=S_n$};

 \end{tikzpicture}
 \end{center}
 \caption{The constructed Roy $(n+1)-$system on $[T_-,T_+]$.}\label{figA}
 \end{figure}
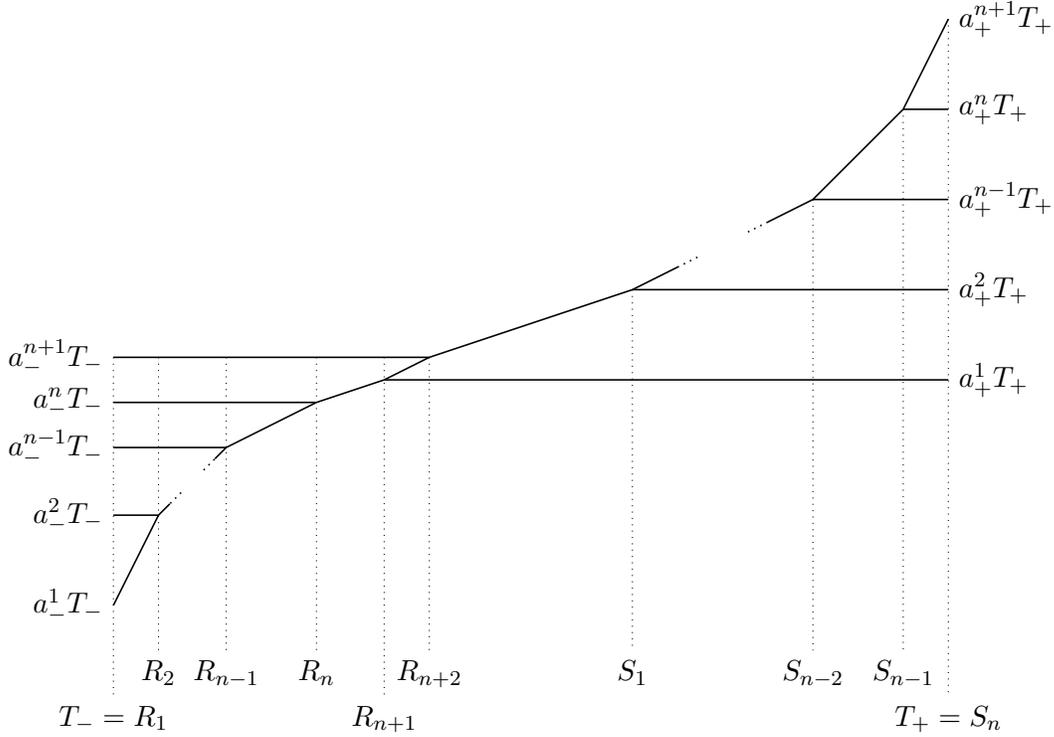

\medskip

 We conclude this section with the following statement.  It provides keys   estimates for  $\bP(t) $ with $t \in  [T_-, T_+]$.

 \begin{lemma}\label{PropBlock}
 The Roy $(n+1)$--system $\bP: [T_-, T_+]\rightarrow \RR^{n+1}$ constructed above satisfies:
 \begin{eqnarray}
    \min_{t\in[T_-, T_+]}\left( \frac{t}{n+1} - P_1(t)\right) &=&  \gamma \label{liminf2}\\[2ex]
   \max_{t\in[T_-, T_+]} \left( \frac{t}{n+1} - P_1(t)\right) &\ge & \left( \frac{1}{n+1} - a_+^1\right)T_+ \label{limsup2}\\[3ex]
    \min_{t\in[T_-, T_+]} \frac{P_1(t)+ \cdots + P_{d}(t)}{t}&=&\min\left\{\sum_{j=1}^d a_-^j, \sum_{j=1}^d a_+^j\right\} \quad(1\le d\le n). \label{d-exponents2}
 \end{eqnarray}
 \end{lemma}
 \begin{proof}
 By construction the derivative of the function $P_1$ is strictly greater than $1/(n+1)$ on the interval $[T_-, R_{n+1}]$ and is $0$ on the interval $[R_{n+1}, T_{+}]$. Here and throughout, by the  derivative of a piecewise linear function on a given interval, we mean the derivative on the union of subintervals on which the derivative exists.  It follows that on the interval $[T_-, T_{+}]$, the local minimum  of the function $f: t \to f(t) := t/(n+1)-P_1(t)$ is achieved at $t=R_{n+1}$.  In other words, the  minimum of $f(t)$ on $[T_-, T_{+}]$ is equal to
\[\frac{R_{n+1}}{n+1}-P_1(R_{ n+1})=  a_-^{n+1}T_- -  n\gamma -  a_+^1T_+   \stackrel{\eqref{def-t}}{=}     \gamma  \, .\]
This shows that $\bP$ satisfies \eqref{liminf2}.  On the other hand, it is easily seen that the maximum of $f(t)$ on $[T_-, T_{+}]$ is achieved at either $t=T_-$ or $t=T_+$.  Thus,
$$ {\rm l.h.s. \ of \ } \eqref{limsup2}    \,  \ge    \,  \frac{T_+}{n+1}-P_1(T_+)=  \frac{T_+}{n+1}-  a_+^1 T_+    \,    $$
and this shows that $\bP$ satisfies \eqref{limsup2}.    It remains to prove \eqref{d-exponents2}.   For simplicity, we let $Q_d:=P_1+\cdots+P_d$  $(1\le d\le n)$ and note that to determine when $Q_d(t)/t$ attains its minimum on  $[T_-, T_{+}]$, it suffices to study the function $$D: t \to D(t):=Q'_d(t)t-Q_d(t) \,  .$$
 On each connected open interval where $\bP$ is differentiable,  it is easily verified that the derivative  $D^{'}(t) = 0$ and so $D(t)$ is constant. Hence, it suffices to study the quantities  $D^*(R_j)$ and $D^*(S_j)$  for $1\le j\le n$. This we now do  systematically.   Recall that for a piecewise continuous function $D$,
\[D^+(t):=\lim_{\substack{ s\rightarrow t,\\ s>t}} D(s), \quad\text{ and }\quad D^-(t):=\lim_{\substack{s\rightarrow t,\\ s<t}} D(s).\]

\begin{itemize}
  \item For $1\le j \le d$, on the interval $(R_j, R_{j+1})$,  the derivative of $Q_d$ equals $1$.  Hence $D(t)$ is positive on this interval.
  \item For $d+1\le j\le n$, on the interval $(R_j, R_{j+1})$, the derivative of $Q_d$ equals $d/j$. A direct computation shows that
  \begin{align*}
  jD^+(R_j) &= dR_j-jQ_d(R_j)\\[1ex]
  &= d\left(a_-^{n+1} + \cdots + a_-^{j+1} +ja_-^{j}\right)T_- -jda_-^{j}T_-  \\
  &>0
\end{align*}
Hence $D(t)$ is positive on this interval.
\item For $0\le j\le d-1$, on the interval $(S_j, S_{j+1})$, the derivative of $Q_d$ equals $(d-j-1)/(n-j)$. A direct computation shows that
\begin{align*}
    (n-j)D^-(S_{j+1})
  &=(d-j-1)S_{j+1}-(n-j)Q_d(S_{j+1})\\[2ex]
  &=(d-j-1)\left(a_+^{1} + \cdots + a_+^{j+1} +(n-j)a_+^{j+2}\right)   \,  T_+\\
  &\qquad \ -(n-j)\left(a_+^{1} + \cdots + a_+^{j+1} +(d-j-1)a_+^{j+2}\right) \,  T_+  \\[1ex]
  &= (d-n-1)\left(a_+^{1} + \cdots + a_+^{j+1} \right)  \,  T_+\\
  &<0
\end{align*}
Hence $D(t)$ is negative on this interval.
\item For $d\le j\le n-1$, on the  interval $(S_j, S_{j+1})$, the derivative of $Q_d$ equals $0$. Hence $D(t)$ is negative on this interval.
\end{itemize}
In conclusion, the function $Q_d(t)/t$ increases on the interval $(R_1, R_{n+1})$ and decrease on the interval $(S_0, S_n)$. As $Q_d(t)/t$ is monotonic on $(R_{n+1}, S_0)$, the minimum is thus  attained at either $t=T_-$ or $t=T_+$. Hence,
$$ {\rm l.h.s. \ of \ } \eqref{d-exponents2}    \,  =   \,  \min\left\{ \frac{Q_d(T_-)}{T_-}, \frac{Q_d(T_+)}{T_+} \right\} =  \min\left\{\sum_{j=1}^d a_-^j, \sum_{j=1}^d a_+^j\right\}   \,    $$
and this shows that $\bP$ satisfies \eqref{d-exponents2} which in turn completes the proof of the lemma.
\end{proof}

\subsection{The local construction on blocks. \label{blocksec}}

In this section, will will exploit the generic construction presented in  \S\ref{Blocks}  to essentially  prove a `local' version of  Lemma~\ref{PropSyst}.  More precisely,   we will construct a family of Roy $(n+1)-$systems $\bP^{\delta}$ on certain subintervals $I$ of $[0,\infty)$   all satisfying the properties of Lemma~\ref{PropBlock}.  Here $\delta\in[0,1/(32n^2))$ is a parameter and for $\delta'\neq \delta$, we show that the intervals $I$ can be chosen so that the Roy $(n+1)-$systems $\bP^{\delta}$ and $\bP^{\delta'}$ on $I$  are mutually non-equivalent.  The construction consists of five short steps. Throughout,  $ n \ge 2$ and
   $\tau_{0}, \ldots, \tau_{n-1} \in [0, \infty]$  are   the real numbers appearing in  Theorem~\ref{thm-main} satisfying  \eqref{gugdsv}.\\

   \noindent {\em Step 1.\ } For $1\le i\le n-1$ and $1\le j\le n+1$, let
 \begin{eqnarray*}
\alpha^{i,j}:=\begin{cases}
                 i^{-1}(1+\tau_{n-i})^{-1}  &  \ \mbox{if }   \ j\le i \\[1ex]
                 (1+\tau_{n-i-1})^{-1}-(1+\tau_{n-i})^{-1}  & \  \mbox{if } \ j=i+1 \\[1ex]
                 (n-i)^{-1}\tau_{n-i-1}(1+\tau_{n-i-1})^{-1}  & \ \mbox{if } \ j>i+1
               \end{cases}
\end{eqnarray*}
where $(1+\tau_{j})^{-1}=0$ and $\tau_j(1+\tau_{j})^{-1}=1$ if $\tau_{j}=\infty$. The following statement summarises useful properties of the associated   sequence
$(\alpha^{i,j})$ that we shall later exploit.
  \\

\begin{lemma} \label{niceone} Let $(\alpha^{i,j})$ be given as above, then
  \begin{itemize}
\item[(a)] for any  $1\le i\le n-1$, $\sum_{1\le j\le n+1} \alpha^{i, j}=1$,
\item[(b)] for any $1\le i\le n-1$ and $1\le j\le j'\le n+1$, $\alpha^{i,j}\le \alpha^{i,j'}$,
\item[(c)] $\alpha^{i,1}+\cdots+\alpha^{i,j}\ge (1+\tau_{n-j})^{-1}$ with equality holds when $j=i, i+1$.
  \end{itemize}
\end{lemma}
\begin{proof}
  Part (a) follows directly from the definition.  To prove the other parts,  for $1\le i\le n$ let  $$\theta_i=(1+\tau_{n-i})^{-1} \, . $$ Then  it follows that \eqref{gugdsv} is equivalent to
 \[\frac{(n-d+1) \, \theta_{n-d}}{n-d}  \ \le  \ \theta_{n-d+1} \ \le  \ \frac{1+d\theta_{n-d}}{d+1}   \qquad  \forall \quad  1\le d\le n-1  , \]
 which in turn is equivalent to
 \begin{equation}\label{EquivGUGD}
 \frac{\theta_i}{i}  \ \le \  \frac{\theta_{i+1}}{i+1} \qquad\text{ and }\qquad \frac{1-\theta_i}{n+1-i} \  \le  \  \frac{1-\theta_{i+1}}{n-i}  \qquad  \forall \quad 1\le i\le n-1.
 \end{equation}

 \noindent To prove part (b), it suffices to show that
 \[\frac{\theta_i}{i}  \ \le \  \theta_{i+1}-\theta_i  \ \le  \  \frac{1-\theta_{i+1}}{n-i} \qquad  \forall \quad   1\le i\le n-1 \, . \]
 This follows directly from \eqref{EquivGUGD}.  It remains to part (c). When $j\le i$, on appropriately  iterating the first inequality of  \eqref{EquivGUGD}, it follows that
 \[\alpha^{i,1}+\cdots+\alpha^{i,j} \ = \ \frac{j \, \theta_i}{i} \ \ge \  \theta_j.\]
 When $j=i$ or $j=i+1$, the statement with equality is easily checked. When $j>i+1$, on appropriately  iterating the second inequality of  \eqref{EquivGUGD},  it follows that
\begin{eqnarray*}
\alpha^{i,1}+\cdots+\alpha^{i,j} \; = \; \theta_{i+1}+\frac{(j-i-1)(1-\theta_{i+1})}{n-i} & =  & 1-\frac{(n+1-j)(1-\theta_{i+1})}{n-i} \\[2ex]
& \ge   &   1-(1-\theta_j)  \ = \ \theta_j.
\end{eqnarray*}
 This completes the proof of the lemma.
\end{proof}

\medskip

\noindent {\em Step 2.\ } Having chosen the sequence $(\alpha^{i,j})$ as above, the second step involves  choosing  a sequence of positive real numbers
\begin{equation} \label{godseq} \left\{\beta_k^{i,j}:\  1\le i\le n-1,\  1\le j\le n+1,\  k\ge 1\right\}
\end{equation}
such that for any $k\ge 1$ and $1\le i\le n-1$:

\begin{eqnarray}
  \sum_{1\le j\le n+1} \beta_k^{i,j}&=& 1      \label{sum-a}\\[1ex]
  \beta_k^{i,j+1}-\beta_k^{i,j} &\ge& \frac{1}{4n^2k}  \quad\text{ for all }  1\le j\le n, \label{difference-a}\\[2ex]
  \beta_k^{i,1}&\in& \left[\frac{1}{(k+1)(n+1)}, \frac{k}{(k+1)(n+1)}\right],   \label{def-alpha}\\[2ex]
  \beta_k^{i,n+1}&\in& \left[\frac{k+3}{(k+1)(n+1)}, 1-\frac{1}{(k+1)(n+1)}\right] \label{def-alpha2}\\[3ex]
   \lim_{k\rightarrow \infty} \beta_k^{i,j} &=& \alpha^{i,j}  \qquad  \forall \quad  1\le j\le n+1 .  \label{limit-a}
\end{eqnarray}\\

\noindent Note that parts (a) and (b) of Lemma~\ref{niceone}  guarantees the existence of such a sequence.   For instance, they imply that  for any $1\le i\le n-1  $
$$
\alpha^{i,1} \leq 1/n \qquad {\rm and } \qquad \alpha^{i,n+1} \ge 1/n \, .
$$
Thus  the conditions   \eqref{def-alpha}, \eqref{def-alpha2} and \eqref{limit-a} are  compatible. \\

\noindent {\em Step 3.\ }  Now, the third step  is to let
\begin{equation}\label{def-t2sv} T_1:=128n^4\gamma
\end{equation}
and  then define inductively $T_k^i$ for $k\ge 1$ and $1\le i\le n-1$ as follows:
\begin{equation}\label{def-t2}
  T_k^1\, :=  \, T_k \, ,  \ \quad \beta_k^{i+1, 1}T_k^{i+1} \, :=  \, \beta_k^{i, n+1}T_k^i-(n+1)\gamma \, , \ \quad T_k^n\, := \, T_{k+1},
\end{equation}
where we set
\begin{equation}\label{def-beta-n}
  \beta_k^{n, j}=\beta_{k+1}^{1,j}    \qquad  \forall \quad   k\ge 1 \quad {\rm and } \quad  1\le j\le n+1.
\end{equation}

\noindent Observe that for any $k\ge 1$ and $ 1\le i\le n-1$, it follows via  \eqref{def-alpha},    \eqref{def-alpha2} and  \eqref{def-t2} that
\begin{eqnarray}  \label{needed}
  T_{k}^{i+1} &=& (\beta_{k}^{i+1,1})^{-1}\left( \beta_k^{i, n+1} \, T_k^{i}-(n+1)\gamma\right)  \nonumber \\[2ex]
   &\ge& \frac{(k+2)(k+3)}{(k+1)^2}  \, T_k^i-   \frac{(n+1)^2(k+2)\gamma}{k+1} \nonumber  \\[2ex]
   &>& \frac{(k+2)^2}{(k+1)^2}\, T_k^i+\left(\frac{T_k^i - (n+1)^2(k+2)\gamma}{k+1}\right)  \, .
\end{eqnarray}

\noindent In turn, on arguing by induction, it follows that for any $k\ge 1$:\\
\begin{equation}\label{incsv}
  T_k^{i+1} \  >  \ T_k^{i} \qquad  \forall \quad  1\le i\le n-1
\end{equation}
and
\begin{equation}\label{e-value-t}
  T_k \ \ge \ 32 \,  n^4(k+1)^{2}\gamma.
\end{equation}

\noindent Indeed,  let $k=1$.   Then  \eqref{e-value-t} holds in view of \eqref{def-t2sv}. To prove \eqref{incsv} we use induction on $i$. With this and \eqref{needed} in mind, when $i=1$   it follows via $\eqref{e-value-t}$ and the fact that $T_1^1 :=T_1$, that
\begin{equation}  \label{needed2}
T_1^1 - (n+1)^2 3 \gamma  \;  \ge  \;  128n^4 \gamma  \, - \, 3 (n+1)^2 \gamma   \, > \, 0  \, .
\end{equation}
Hence,  \eqref{needed} implies that $ T_1^2 > T_1^1 $. In other words,   \eqref{incsv} holds for $i=1$.   So suppose \eqref{incsv}  holds for $i$ with $ i \le n-2$. Then,  it follows via  \eqref{needed2}  that
$$
T_1^{i+1} - (n+1)^2 3\gamma \;  > \; T_1^1 - (n+1)^2 3 \gamma  \;  \ >   \; 0
$$
and so \eqref{needed} implies that $T_1^{i+2} > T_1^{i+1} $.  This shows that \eqref{incsv}  holds with $k=1$.    Now  assume  that \eqref{incsv} and \eqref{e-value-t} holds for $k$.   Then, it follows via   \eqref{needed} with $i=n-1$  and  the fact that $ T_k^n\, := \, T_{k+1}$ and $ T_k^1\, := \, T_{k}$, that
 \begin{eqnarray*}
  T_{k+1} &>& \frac{(k+2)^2}{(k+1)^2}\, T_k^{n-1}+\left(\frac{T_k^{n-1} - (n+1)^2(k+2)\gamma}{k+1}\right)\\[2ex]
   &>& \frac{(k+2)^2}{(k+1)^2}\, T_k+\left(\frac{T_k - (n+1)^2(k+2)\gamma}{k+1}\right) \\[2ex]
   &>& 32 n^4(k+2)^{2}\gamma  \, .
\end{eqnarray*}
This  shows that \eqref{e-value-t} holds for $k+1$ and we now use this to show that \eqref{incsv} holds for $k+1$.  With this in mind, when $i=1$  it  follows that
\begin{equation}  \label{needed22}
T_{k+1}^1 - (n+1)^2(k+3)\gamma \;  \ge  \;  32 n^4(k+2)^{2}\gamma  \, - \, (n+1)^2(k+3)\gamma   \, > \, 0
\end{equation}
 and so \eqref{needed} implies that $ T_{k+1}^2 > T_{k+1}^1 := T_{k+1}$. In other words,   \eqref{incsv} holds for $k+1$ with  $i=1$.   So suppose \eqref{incsv}  holds for $k+1$ with $ i \le n-2$. Then it follows via  \eqref{needed22}  that
$$
T_{k+1}^{i+1} - (n+1)^2 3\gamma \;  > \; T_{k+1}^1 - (n+1)^2(k+3)\gamma  \;  \ >   \; 0
$$
and so \eqref{needed} implies that $T_{k+1}^{i+2} > T_{k+1}^{i+1} $. This thereby completes the inductive step and hence  establishes   \eqref{incsv} and \eqref{e-value-t} for all $k \ge 1$. \\

 Now, with the generic construction of $\S\ref{Blocks} $ in mind,  for $k\ge 1$, $1\le i \le n-1$, let
 $$T_-=T_k^i     \qquad {\rm and }  \qquad T_+=T_k^{i+1}  \, , $$
 and  for  $1\le j \le n+1$,  let
 $$a_-^j  = \beta_k^{i,j}   \qquad {\rm and }  \qquad  a_+^j = \beta_k^{i+1,j}  \, . $$ Then, it is readily verified on using  \eqref{sum-a}, \eqref{difference-a},  \eqref{limit-a}, \eqref{def-t2}, \eqref{incsv} and \eqref{e-value-t} that conditions  \eqref{def-t}, \eqref{def-sum} and \eqref{def-dist} are satisfied.
The upshot is that the construction described within \S\ref{Blocks} is applicable and  gives rise to a
Roy $(n+1)$--system $\bP: [T_k^i, T_k^{i+1}]\rightarrow \RR^{n+1}$  associated with the sequences $( \beta_k^{i,j})$  defined via \eqref{godseq} and $(T_k^i)$  defined via \eqref{def-t2}.  Moreover, for each $k  \ge 1$, $1\le i \le n-1$,  the  Roy $(n+1)$--system $\bP$ on  the interval $[T_k^i, T_k^{i+1}]$  satisfies Lemma~\ref{PropBlock}.

\begin{remark}
 As we shall see in the next section, it is not difficult to extend this local statement  to a Roy $(n+1)$--system $\bP$ on  the interval $[0, \infty)$  that satisfies Lemma~\ref{PropSyst}.  Note that this would suffice if all we wanted to show was that the sets appearing in Theorem~\ref{thm-main}  are non-empty rather than {continuum}.
\end{remark}

 \medskip

 \noindent {\em Step 4.\ } The  fourth step involves perturbing the above construction  of the  Roy $(n+1)$--system $\bP$ on $ [T_k^i, T_k^{i+1}]$ by  a parameter $\delta$ in such a way that:
  \begin{itemize}
    \item the properties of Lemma~\ref{PropBlock} are satisfied  for the perturbed  Roy $(n+1)-$system $\bP^{\delta}: [T_k^i, T_k^{i+1}]\rightarrow \RR^{n+1}\!$,  and
    \item  for $\delta'\neq \delta$, the perturbed Roy $(n+1)-$systems $\bP^{\delta}$ and  $\bP^{\delta'}$ are mutually non-equivalent (see Definition~\ref{neRS}).
  \end{itemize}

\noindent With this in mind,  let   $( \beta_k^{i,j})$  and $(T_k^i)$ be the sequences  given by   \eqref{godseq} and  \eqref{def-t2} respectively, and let
\begin{equation} \label{deltabd}
\delta\in \big[0, 1/32n^2\big)  \, .
\end{equation}  Now define the  new sequence
\begin{equation} \label{vgodseq} \left\{\beta_k^{i,j}(\delta):  \  1\le i\le n-1,\  1\le j\le n+1,\  k\ge 1\right\}
\end{equation}
by setting, for any  $k\ge 1$ and $1\le i\le n-1$:
\begin{eqnarray}
   \beta_1^{1,1}(\delta) &:=&  \beta_1^{1,1},  \label{justincase} \\[2ex]
  \beta_k^{i,n+1}(\delta) &:=&  \beta_k^{i,n+1}+\frac{\delta}{k}, \label{e-delta1}\\[2ex]
\beta_k^{i+1,1}(\delta)T_k^{i+1} &:=& \beta^{i, n+1}_k(\delta)T_k^i-(n+1)\gamma, \label{e-delta2} \\[2ex]
  \beta_k^{i,1}(\delta)+\beta_k^{i, 2}(\delta)+\beta_k^{i,n+1}(\delta) &:=&
  \beta_k^{i,1}+\beta_k^{i, 2}+\beta_k^{i,n+1} \label{e-delta3},\\[2ex]
  \beta_k^{i,j}(\delta)&:=& \beta_k^{i,j}  \qquad  \forall \quad  3\le j\le n  \, .  \label{e-delta4}
\end{eqnarray}
Also, in line with \eqref{def-beta-n}, we let
\[  \beta_k^{n, j}(\delta) \ := \ \beta_{k+1}^{1,j}(\delta) \qquad  \forall \quad   k\ge 1 \quad {\rm and } \quad  1\le j\le n+1. \]
Clearly, the  sequences  $(\beta_k^{i,j}(\delta)) $ and $(\beta_k^{i,j})$ coincide  when $\delta=0$. An immediate consequence of \eqref{sum-a}, \eqref{e-delta3} and \eqref{e-delta4} is that
\begin{equation}\label{sum-delta}
  \sum_{j=1}^{n+1} \beta_k^{i,j}(\delta)=1.
\end{equation} \\
Also note that in view of \eqref{def-t2}, \eqref{e-delta1} and \eqref{e-delta2}, we have that
\[\left(\beta_k^{i+1,1}(\delta)-\beta_k^{i+1,1}\right)T_k^{i+1}  \, =  \,
\left(\beta_k^{i,n+1}(\delta)-\beta_k^{i,n+1}\right)T_k^i  \, = \, \frac{\delta}{k}T_k^i \, ,\]
from which it follows that
\[\beta_k^{i,1}  \ \le  \  \beta_k^{i,1}(\delta)  \ \le  \  \beta_k^{i,1}+\frac{\delta}{k} \qquad  \forall \quad 2\le i\le n-1  \] and
\begin{equation}\label{lalji}
\beta_{k+1}^{1,1}   \ \le  \  \beta_{k+1}^{1,1}(\delta)  \ \le   \ \beta_{k+1}^{1,1}+\frac{\delta}{k} \, .
\end{equation}
To sum up, for all $k\ge 1$ and $1\le i\le n-1$  we have  that
\begin{equation}\label{e-ine-20}
 \beta_k^{i,1}\le \beta_k^{i,1}(\delta)\le \beta_k^{i,1}+\frac{2\delta}{k}.
\end{equation}
Combining \eqref{e-ine-20} and \eqref{e-delta3}, we get for all $k\ge 1$ and $1\le i\le n-1$,
\begin{equation}\label{e-ine-21}
  \beta_k^{i, 2}-\frac{3\delta}{k}\le \beta_k^{i, 2}(\delta)\le \beta_k^{i, 2}.
\end{equation}\\

 Now, with the generic construction of $\S\ref{Blocks} $ in mind,  for $k\ge 1$, $1\le i \le n-1$, let
 $$ T_-=T_k^i     \qquad {\rm and }  \qquad T_+=T_k^{i+1}  \, , $$
 and  for  $1\le j \le n+1$,  let
 $$a_-^j  = \beta_k^{i,j}(\delta)   \qquad {\rm and }  \qquad  a_+^j = \beta_k^{i+1,j} (\delta) \, . $$ Then, it is readily verified
that condition \eqref{def-t} follows from \eqref{e-delta2}  and that  condition \eqref{def-sum} follows from \eqref{limit-a} and \eqref{sum-delta}. To show \eqref{def-dist},  first note that for all $k\ge 1$, $1\le i\le n-1$ and $j\ne 1$, on using \eqref{e-delta1}, \eqref{e-delta4} and  \eqref{e-ine-21}, it follows that
\[\left( \beta_k^{i,j+1}(\delta)-\beta_k^{i,j}(\delta) \right)  T_k  \ge \left(\beta_k^{i,j+1}-\beta_k^{i,j} \right) T_k   \, . \]
We have already shown that the right hand side satisfies \eqref{def-dist}.
When $j=1$, it is readily verified, on using   \eqref{difference-a},   \eqref{e-value-t}, \eqref{deltabd},  \eqref{e-delta1} and  \eqref{e-ine-21}, that for all $k\ge 1$ and $1\le i\le n-1$
\[\left(\beta_k^{i,2}(\delta)-\beta_k^{i,1}(\delta)\right)T_k^i\ge \left(\beta_k^{i,2}-\beta_k^{i,1}-\frac{4\delta}{k}\right)T_k^i \ge \frac{T_k}{8n^2k}\ge 4n^2\gamma.\]
The upshot is that the construction described within \S\ref{Blocks} is applicable and  gives rise to a
Roy $(n+1)$--system $\bP^\delta: [T_k^i, T_k^{i+1}]\rightarrow \RR^{n+1}$  associated with the constant $\delta$ satisfying \eqref{deltabd} and sequences $( \beta_k^{i,j}(\delta))$  defined via \eqref{vgodseq} and $(T_k^i)$  defined via \eqref{def-t2}.  Moreover, for each $k  \ge 1$, $1\le i \le n-1$,  the  Roy $(n+1)$--system $\bP^\delta$ on  the interval $[T_k^i, T_k^{i+1}]$  satisfies Lemma~\ref{PropBlock}.
\\

 \medskip

 \noindent {\em Step 5.\ }   It remains to show that the  Roy $(n+1)$--systems constructed in Step~4 are mutually non-equivalent.  This is easily done. Let $\delta$ and $\delta'$ satisfy \eqref{deltabd} and suppose $\delta'\ne\delta$.
 Then it is readily verified, that
\[\left|P_{n+1}^{\delta'}(T_k)-P_{n+1}^{\delta}(T_k)\right|   \  \stackrel{\eqref{e-delta1}}{=} \   \frac{|\delta'-\delta|  \; T_k}{k}   \ \stackrel{\eqref{e-value-t}}{\ge}  \  |\delta'-\delta|32n^4k\gamma  \ > \ 2C_n \,  \]
 for all $k> k_0$ sufficiently large.  By definition, this implies that for any $k>k_0$, the Roy $(n+1)-$systems $\bP^{\delta}$ and $\bP^{\delta'}$ on the interval  $[T_k, T_{k+1}]$ are mutually non-equivalent.

\subsection{Proof of Lemma \ref{PropSyst}.}\label{PfPropSyst}

The proof of the   Lemma \ref{PropSyst} will follow on  extending  the local construction of the  Roy $(n+1)-$systems $\bP^{\delta}$ on the  intervals $[T_k^i, T_k^{i+1}]$ presented in  \S\ref{blocksec}  to  the interval $[0, \infty)$.  With this in mind, for $\delta$ satisfying  \eqref{deltabd} and $k\ge 1$,  $1\le i \le n-1$,  let us denote by $\bP^{\delta}_{k,i}$ the Roy $(n+1)-$system  on  $[T_k^i, T_k^{i+1}]$.  Now observe that
\[[T_1, \infty)=\bigcup_{k\ge 1}   \  \bigcup_{1\le i\le n-1} \,  [T_k^i, T_k^{i+1}].\]
where  $T_1$  is given by \eqref{def-t2sv}.  It therefore follows that the continuous piecewise linear map $\bP^{\delta}=(P_1^{\delta},\ldots, P_{n+1}^{\delta}): [T_1,\infty) \rightarrow \RR^{n+1}$ given by
$$
\bP^{\delta}(t) :=  \bP^{\delta}_{k,i}(t)  \qquad  {\rm for }  \quad   t \in [T_1,\infty)  \, ,
$$
is a  Roy $(n+1)$--system on $[T_1,\infty)$.
It remains to extend   $\bP^\delta$ to the interval $[0, T_1]$.  For this let
\[  S_{d+1} := \left(\beta_1^{1, 1}(\delta) + \cdots + \beta_1^{1, d}(\delta) +(n+1-d)\beta_1^{1, d+1}(\delta)\right)T_1 \qquad  \forall \quad   0\le d \le n,\]
and let
$$
P_j^\delta(0):= 0 \qquad  \forall \quad  1\le j\le n+1  \, .
$$
On the interval $[0,S_{1}]$, let the $n+1$ components $P_1^\delta, \ldots , P_{n+1}^\delta$ coincide and have slope $1/(n+1)$. It follows that
$$
P_j^\delta(S_1):= \beta_1^{1, 1}(\delta) T_1 \qquad  \forall \quad  1\le j\le n+1  \, .
$$
For $1\le d \le n$, on the interval $[S_{d},S_{d+1}]$, let the $n+1-d$ components $P_{d+1}^\delta, \cdots, P_{n+1}^\delta$ coincide and have slope $1/(n+1-d)$ while
the components $P_1^\delta \ldots , P_{d}^\delta$ have slope $0$  and
\begin{eqnarray*}
P_j^\delta(S_{d+1}):=\begin{cases}
                 \beta_1^{1, j}(\delta) \, T_1 & \mbox{if }   \ 1 \le j\le d \\[1ex]
                 \beta_1^{1, d+1}(\delta) \, T_1   & \mbox{if } \ d+1 \le j \le n+1 \, .
               \end{cases}
\end{eqnarray*}
In particular, by \eqref{sum-delta}  we have  that $ S_{n+1} =  T_1$ and so it follows that
\[P_j^\delta(T_1)=\beta_1^{1, j}(\delta) T_1    \qquad  \forall \quad  1\le j\le n+1   \, . \]
In short, this coincides with left hand side of Figure~1 with  $T_-=T_1$ and $a_-^j = \beta_1^{1,j}(\delta)$ $(1\le j\le n+1)$.  The upshot is that the above construction enables us to extend  in then obvious manner  the Roy $(n+1)$--system on $[T_1,\infty)$ to $[0,\infty)$.   We now show that the Roy $(n+1)-$system $\bP^\delta : [0, \infty)\rightarrow \RR^{n+1}$ satisfies the desired properties of Lemma \ref{PropSyst}.  \\

\noindent  \  \ $\bullet$ \ By construction, \eqref{liminf} follows directly from \eqref{liminf2}.\\

\noindent  \  \ $\bullet$ \ In view of  \eqref{limsup2},   \eqref{def-alpha},  \eqref{e-value-t} and \eqref{lalji}, it follows that
\begin{align*}
  \limsup_{t \to \infty}  \left( \frac{t}{n+1} - P_1^{\delta}(t)\right)
   &= \limsup_{k \to \infty} \  \max_{1\le i\le n}  \ \max_{t \in[T_k^i,T_{k}^{i+1}]} \left( \frac{t}{n+1} - P_1^{\delta}(t)\right) \\[2ex]
   &\ge \limsup_{k \to \infty}   \ \max_{t \in[T_k^{n-1},T_{k}^{n}]} \left( \frac{t}{n+1} - P_1^{\delta}(t)\right) \\[2ex]
  &\ge \limsup_{k\rightarrow \infty} \left(\frac{1}{n+1}- \beta_{k}^{n,1}(\delta) \right)  T_{k}^n  \\[2ex]
  &\ge \limsup_{k\rightarrow \infty}\left(\frac{1}{n+1}-\beta_{k+1}^{1,1}-\frac{\delta}{k} \right) T_{k+1}\\[2ex]
  &\ge \limsup_{k\rightarrow \infty}\left(\frac{1}{n+1}-\frac{k+1}{(k+2)(n+1)}-\frac{1}{32 n^2 k} \right) T_{k+1}\\[2ex]
  &\ge \limsup_{k\rightarrow \infty} n^3 (k+2) \gamma    \    =  \  \infty  \, .
\end{align*}
This shows that $\bP^{\delta}$ satisfies \eqref{limsup}.\\

\noindent  \  \ $\bullet$ \  Let $1\le d\le n$. Then, in view of  \eqref{d-exponents2}, \eqref{limit-a} and part (c) of Lemma~\ref{niceone}, it follows that
\begin{align*}
 \liminf_{t \to \infty}  \frac{P_1^{\delta}(t)+ \cdots + P_d^{\delta}(t)}{t}
 &= \liminf_{k \to \infty}  \ \min_{1\le i\le n} \ \min_{t \in[T_k^i,T_{k}^{i+1}]} \left(\frac{P_1^{\delta}(t)+ \cdots + P_d^{\delta}(t)}{t}  \right)\\[2ex]
&= \liminf_{k \to \infty} \min_{1\le i\le n-1} \min\left\{\sum_{j=1}^d \beta_k^{i,j}(\delta), \sum_{j=1}^d \beta_k^{i+1, j}(\delta)\right\}\\[2ex]
&= \liminf_{k \to \infty} \min_{1\le i\le n-1} \min\left\{\sum_{j=1}^d \beta_k^{i,j}, \sum_{j=1}^d \beta_k^{i+1, j}\right\}\\[2ex]
&= \min_{1\le i\le n} \alpha^{i,1} + \cdots + \alpha^{i,d} \\[2ex]
&=  \frac{1}{1+\tau_{n-d}}.\\
\end{align*}
This shows $\bP^{\delta}$ satisfies \eqref{d-exponents}.  \\

In the previous section (see Step 5), we have already seen that for any  distinct $\delta$ and $\delta'$ satisfying \eqref{deltabd}, the Roy $(n+1)-$systems $\bP^{\delta}$ and $\bP^{\delta'}$ on $[T_k, T_{k+1}]$  are mutually non-equivalent for all $k$ sufficiently large.    Consequently, there are {continuum} many non-equivalent Roy $(n+1)-$systems $\bP^{\delta}$ on $[0,\infty)$ that satisfies the conditions of Lemma~\ref{PropSyst}.   This completes the proof.  \qed \\

\subsection{Further comment\label{FC} }

Recall, that once we have constructed in \S\ref{PfPropSyst}  the  collection of non-equivalent Roy $(n+1)-$systems $\bP^{\delta}$  satisfying  Lemma~\ref{PropSyst}, the key ingredient towards  establishing our main result is Theorem \ref{thm-roy}.  In short, the latter guarantees that each system in our collection gives rise to a distinct point  $\bx\in \RR^n$ such that
  $$\|\bL_{\bx} - \bP^{\delta}\| \le  C_n       $$
  on $[0,\infty)$. In turn, it is not difficult to show that such $\bx$
 satisfies the desired Diophantine properties associated with Theorem~\ref{thm-main}.   Now with this in mind, let $n \in \N$ and $\cW$ be a collection of Roy $(n+1)-$systems $\bP:[0,\infty) \rightarrow \R^{n+1} $. In \cite[Theorem 2.3]{VarPrinc}, it is shown that if  $\cW$ satisfies the so called `closed under finite perturbations' hypothesis,   then one is  able to compute the Hausdorff dimension of the set
\[\left\{\bx\in \RR^n: \|\bL_\bx-\bP\|<\infty \text{ for some } \bP \in \cW\right\}.\]
However, it is not clear to us  whether the methods in \cite{VarPrinc} can be adapted to give a result on the Hausdorff dimension of the set
\[\left\{\bx\in \RR^n: \|\bL_\bx-\bP\|<C \text{ for some } \bP \in \cW\right\} \,  \]
for any fixed constant $C>0$.  Such a result would potentially  enable us to replace ``{continuum}'' by ``full Hausdorff dimension'' in the statement of Theorem~\ref{thslv1}.  We state this highly desired strengthening formally as an open problem.

\medskip

\begin{problem}\label{prob-HD}
Suppose that $n   \geq  2$. Prove that
$$
\dim \, \FS_n   = n  \, .
$$
\end{problem}

\noindent In \S\ref{Sec_A4} below we consider the natural  generalisation of the above problem to the setting of weighted Diophantine approximation for systems of linear forms.

A more subtle version of Problem~\ref{prob-HD} would be to understand how the removal of $\Bad_n\cup\Sing_n$ affects the quantitative form of the set  $\DI_n$, that is to estimate from below the Haudorff dimension of $\FS(\eps)$ , where for $0< \eps < 1$
$$
\FS_n(\eps) := \DI_n(\eps)  \setminus  (\Bad_n  \ \cup \  \Sing_n )  \, .
$$
Clearly, by definition, $\dim\FS_n(\eps) $ cannot be larger than $\dim \DI_n(\eps)$ and by a recent result due to Kleinbock $\&$ Mirzadeh \cite{KleMir}, the latter is strictly smaller than $n$ for all $\eps\in(0,1)$. Note that Theorem~\ref{thm-main},  guarantees that $$\FS_n(\eps) \, \neq  \, \varnothing  \,  \qquad \forall \ \  \eps \in (0,1) . $$

\section{Intermediate Diophantine sets revisited}\label{App}

The main goal of this section is to define  the notion of $d$-Dirichlet improvable points in $\R^n$ and investigate the relationship between them and the classical notions of simultaneous ($d=0$) and dual ($d=n-1$) Dirichlet improvable points.  We will explore two different approaches:
\begin{itemize}
  \item[(i)]
an algebraic approach using multilinear algebra and thus developing the ideas of Laurent \cite{MLwd}, and
\item[(ii)]
a geometric approach using Minkowski minima and thus developing the ideas of Schmidt $\&$ Summerer \cite{SSfirst,SS} and Roy \cite{Roy} on the parametric geometry of numbers leading to Lemma~\ref{DS} of \S\ref{sec-reform}.
\end{itemize}
\noindent The key ingredient required for achieving this lies in being  able to state an appropriate optimal Dirichlet type theorem (see Remark~\ref{bull} in \S\ref{setupmain}). This will be accomplished in \S\ref{noway} via the framework of multilinear algebra and in \S\ref{DTTGN} via the framework of the parametric geometry of numbers.
In the process of describing the setup leading to the Dirichlet type theorem of \S{\ref{DTTMA}}, we will take the opportunity to first revisit the intermediate badly approximable and singular sets in order to fill in the details of the arguments (cf. Remark~{\ref{ohyaohya}})  leading to {\eqref{equ-equi}}  and
{\eqref{equ-equiSING}}. Recall, that these equivalences are used in the ``classical'' proof of Proposition~{\ref{prop-equiv}} given in \S{\ref{sec-reform}} showing that the   intermediate $d$-badly approximable sets $\Bad_n^d$ (resp. the $d$-singular sets $\Sing_n^d$) are equivalent.  Moreover, we will provide an alternative ``dynamical'' proof of the proposition that avoids appealing to {\eqref{equ-equi}} and {\eqref{equ-equiSING}}.

\subsection{The algebraic approach}  \label{noway}

To proceed, we recall notions and results from  multilinear  algebra. Let $n\ge 2$ and $0\le d\le n-1$. First, we endow the linear space $\RR^{n+1}$ with the usual inner product and let $\{\be_{i}\}_{1\le i \le n+1}$ be the standard orthonormal basis. Then the wedge product $\wedge^{d+1}\RR^{n+1}$ is also equipped with an inner product with an orthonormal basis given by
\[\left\{\be_{i_1}\wedge \cdots \wedge \be_{i_{d+1}}: 1\le i_1\le \ldots \le i_{d+1}\le n+1\right\}.\]

\noindent Note that, under this basis, we identify $\wedge^{d+1}\RR^{n+1}$ with $\RR^{{n+1 \choose d+1}}$. For any $\bX\in \wedge^{d+1}\RR^{n+1}$, we set $|\bX|$ and $\|\bX\|$ to be the Euclidean norm and maximal norm (with respect to the basis given above), respectively.
 $\bX\in \wedge^{d+1}\RR^{n+1}$ is called \emph{decomposable} if and only if there exists $\bv_1, \ldots, \bv_{d+1}\in \RR^{n+1}$ such that $\bX=\bv_1\wedge \cdots \wedge \bv_{d+1}$.  It is worth highlighting that up to an homothety there is a one to one correspondence between decomposable $\bX  \in \wedge^{d+1}\R^{n+1}$  and $d$-dimensional rational affine subspaces of $ \RR^n.$  Indeed, by expressing a $d$-dimensional affine subspace $L \subset \RR^n \subset \PP^n(\RR)$ using homogeneous coordinates, we obtain a unique $(d+1)$-dimensional subspace $V_L$ of $\RR^{n+1}$ satisfying that $L =   \PP(V_L) $. Clearly, $L$ is rational if and only if $V_L$ has a integer basis $\{ \bv_1, \ldots, \bv_{d+1} \} \subset \Z^{n+1}$. By using the Pl\"ucker embedding
\begin{equation}\label{vb+1}
\mathrm{Gr}(d, \PP^n(\RR)) \hookrightarrow \PP(\wedge^{d+1}\RR^{n+1}),\quad L \mapsto \bX_L := \bv_1 \wedge \cdots \wedge \bv_{d+1}
\end{equation}
we obtain the Pl\"ucker coordinates for $L$.
Conversely, given a decomposable multivector $\bX\in\bigwedge^{d+1}\Z^{n+1}\setminus\{0\}$, the associated linear subspace $V_L$ can be expressed as $V_L:=\{\bx\in\R^{n+1}:\bx\wedge\bX=\mathbf{0}\}$ and then $L$ can be obtained by projecting the intersection $V_L\cap\{\bx\in\R^{n+1}:x_{n+1}=1\}$ onto the first $n$ coordinates.
The \emph{height} $H(L)$ of $L$ will be the Weil height, that is
\[H(L)  =  |\bX_L|  \,, \]
{where $\bX_L$ is as in \eqref{vb+1} and $\bv_1, \ldots, \bv_{d+1}$ is a basis of $V_L\cap \Z^{n+1}$.
In other words, $H(L)$ is the covolume of the lattice $V_L\cap \Z^{n+1}$ in $V_L$.
Note that $\bX_L\in \wedge^{d+1}\ZZ^{n+1}\setminus \{0\}$ has coprime coordinates.} Given  $ \bx \in \R^n$, we define the projective distance between $\bx$ and $L$ by
\begin{equation}\label{d_p}
d_p(\bx, L):=\frac{|\bx'\wedge \bX_L|}{|\bx'||\bX_L|},
\end{equation}
where $\bx':=(\bx, 1)\in \RR^{n+1}$. Geometrically, $d_p(\bx, L)$ equals the sine of the smallest angle between $\bx' $ and non-zero vectors in $V_L$. The projective distance $ d_p(\bx, L)$  is easily seen to be
locally (depending on $|\bx|$) comparable to the distance $d(\bx, L)$ defined by \eqref{normdist}  in \S\ref{setupmain} and indeed the distance of $\bx$ from $L$ induced by any norm on $\R^n$, see for instance, \cite[Eq~(3.5)]{Ber12}.



\subsubsection{Showing $ \Sing_n^d \equiv \eqref{equ-equi} $ and  $\Bad_n^d   \equiv \eqref{equ-equiSING}$ }\label{A11}

The following statement provides an algebraic formulation of the sets $\Sing_n^d$ and $\Bad_n^d$ as defined via \eqref{equ-gen} and
\eqref{slv-bad-d} in \S\ref{setupmain}.  Recall,  given $ n \in \N$ and  $ d \in \{0,1, \ldots, n-1\} $,  we let
$$
\omega_{d}:=\frac{d+1}{n-d}  \, .
$$

\begin{lemma}\label{l-bad-sing}
Let $n \ge 2$ and $ 0 \le d \le n-1$.
\begin{itemize}
  \item[{\rm(i)}] $\bx\in \Sing_n^d$ if and only if for any given $\eps \in (0,1)$ and $N >N_0(\bx, \eps)$ sufficiently large, there exists a $\bX\in \wedge^{d+1} \ZZ^{n+1}\setminus\{0\}$, such that
      \begin{equation}\label{e-equation-sing}
        |\bX|\le N     \quad \text{and} \quad |\bx'\wedge \bX|\le \eps N^{-\omega_d}.
      \end{equation}
  \item[{\rm(ii)}]  $\bx\in \Bad_n^d$ if and only if there exists a constant $\eps:=\eps(\bx) \in (0,1)$
      such that for any $N >N_0(\bx)$ sufficiently large, there are no $\bX\in \wedge^{d+1} \ZZ^{n+1}\setminus\{0\}$ satisfying \eqref{e-equation-sing}.
\end{itemize}
\end{lemma}

\begin{remark} \label{corona}
In view of the discussion at the start of this section, it is  straightforward to establish the above reformulation of the sets  $\Sing_n^d$ and $ \Bad_n^d$  under the extra assumption that $\bX\in \wedge^{d+1} \ZZ^{n+1}\setminus\{0\}$ is decomposable.
In view of this,  the key feature of Lemma~{\ref{l-bad-sing}} is that it enables us to remove the decomposable assumption. \end{remark}

\begin{proof}
The `only if part' of both (i) and (ii) follows directly from Remark~\ref{corona}.  Now concentrating on the `if part' of (i),   in view of Remark~\ref{corona},  it suffices to show that  for any given $\eps \in (0,1)$ and $N >N_0(\bx, \eps)$ sufficiently large, there exists a decomposable $\bX\in \wedge^{d+1} \ZZ^{n+1}\setminus\{0\}$ such that \eqref{e-equation-sing} holds.  With this in mind, let $\eps\in (0,1)$, $N >N_0(\bx, \eps)$  and $\bX\in \wedge^{d+1} \ZZ^{n+1}\setminus\{0\}$ satisfying \eqref{e-equation-sing} be as given.
This implies  that the first successive minima of the convex body $\mathcal{C}$ defined by
\begin{equation}\label{slv-compound}\left\{\bX\in \wedge^{d+1}\RR^{n+1} \, : \  |\bX|\le N \, , \quad |\bx'\wedge \bX|\le \eps N^{-\omega_d} \right\}
\end{equation}
is less than $1$. Let
\[U:=\eps^{-d/(d+1)}N^{n/(n-d)} \qquad \text{and} \qquad V:=\eps^{(n+d)/(nd+n)}U^{1/n}=\eps^{1/(d+1)}N^{-1/(n-d)}.\]
Thus, $N=UV^{d}$ and $\eps N^{-\omega_d}=V^{d+1}$. Then, in view of \cite[Lemma 3]{BugLau2}, there exists  $\beta_1=\beta_1(n,d)>1$ such that
\begin{equation}\label{e-compound}
\beta_1^{-1}\widetilde{\mathcal{C}}_{d+1} \ \subset \ \mathcal{C}\  \subset \  \beta_1\widetilde{\mathcal{C}}_{d+1} \, ,
\end{equation}
where $\widetilde{\mathcal{C}}_{d+1}$ is the $(d+1)$-th compound of the convex body $\widetilde{\mathcal{C}}\subset \RR^{n+1}$ defined as
\[\widetilde{\mathcal{C}}:=\left\{\by\in \RR^{n+1}  \, : \ |y_{n+1}|\le U \, , \quad \max_{1\le i\le n}|y_{n+1}x_i-y_i|\le V \right\}.\]
In turn, it follows via  Mahler's theory of compound convex bodies that there exists $\beta_2=\beta_2(n,d)\ge 1$ such that
\begin{equation}\label{e-wed}
 \bX:=\bx_1\wedge \cdots \wedge\bx_{d+1}  \, \in \,  \beta_2 \
\lambda_1\left(\wedge^{d+1}\ZZ^{n+1},\mathcal{C}\right) \, \widetilde{\mathcal{C}}_{d+1}
\end{equation}
where $\bx_i\in \ZZ^{n+1}$ is the integer point at which $\widetilde{\mathcal{C}}$ attains its $i$-th successive minima. Recall,  for a given convex body $\mathcal{C}\subset \wedge^{d+1}\RR^{n+1}$ and $i=1, \ldots,d+1$,  we write $\lambda_i(\wedge^{d+1}\ZZ^{n+1},\mathcal{C})   $ for  the $i$-successive minima  of $ \mathcal{C} $  with respect to the lattice $\wedge^{d+1}\ZZ^{n+1}$.   The upshot of \eqref{e-wed} is that $\bX$ is decomposable and this proves the `if part' of  (i).  The proof of the `if part' of (ii) is similar and we leave the details to the reader.
\end{proof}

Armed with Lemma~\ref{l-bad-sing}, it is relatively straightforward to obtained the sought after statement.

\begin{lemma}\label{l-bad-sing2} Let $n \ge 2$ and $ 0 \le d \le n-1$.
\begin{itemize}
  \item[{\rm(i)}] $\bx \in \Sing_n^d$ if and only if for any  $\delta > 0$ there exists a constant $t_0= t_0({\delta})  > 0 $ such that for all $t \ge  t_0$ inequality \eqref{equ-equiSING} holds; that is
      \begin{equation*}
  \frac{(n-d)t}{n+1} - (L_{\bx,1}(t) + \cdots + L_{\bx,n-d}(t)) \ge \delta.
\end{equation*}
  \item[{\rm(ii)}] $\bx \in \Bad_n^d$ if and only if there exists a constant  $\delta > 0$ such that for all sufficiently large  $t$ inequality \eqref{equ-equi} holds; that is
\begin{equation*}
  \frac{(n-d)t}{n+1} - (L_{\bx,1}(t) + \cdots + L_{\bx,n-d}(t)) \le \delta.
\end{equation*}
\end{itemize}
\end{lemma}

 \begin{proof}
 Lemma \ref{l-bad-sing} implies  that  $\bx \in \Sing_n^d$ if and only if for any  $\delta > 0$ there exists a constant $t_0= t_0({\delta})  > 0 $ such that for all $t \ge  t_0$,
\[\lambda_1\left( \wedge^{d+1}\ZZ^{n+1}, \mathcal{K}_{\bx}^{d+1}(\e^{t}) \right)\le e^{-\delta}\]
where
\[ \mathcal{K}_{\bx}^{d+1}(\e^{t}) := \left\{\bX \in \wedge^{d+1}\RR^{n+1}  \, :  \  | \bX | \le \e^{\frac{(n-d)t}{n+1}}  \,  ,  \quad |\bx'\wedge \bX | \le e^{\frac{-(d+1) t}{n+1}} \right\}.\]
 In view of \eqref{e-compound}, the convex body $ \mathcal{K}_{\bx}^{d+1}(\e^{t})$ is comparable (with implied constants depending on $n$ and $d$ only)  to the $(d+1)$-th compound of the convex body
 \[\mathcal{K}_{\bx}(e^t):=\left\{\by \in \RR^{n+1}\, :  \  |y_{n+1}| \le \e^{\frac{nt}{n+1}} \,  , \quad \max_{1\le i\le n}|y_{n+1}x_i-y_i | \le e^{-\frac{t}{n+1}} \right\}   \, . \]
  Note that
  \[\cC_{\bx}(\e^t)=e^{-\frac{t}{n+1}} \ \mathcal{K}_{\bx}(e^t)^{\vee},\]
  where $ \cC_{\bx}(\e^t) $ is given by  \eqref{def-sc} and $\mathcal{K}_{\bx}(e^t)^{\vee}$ denotes the dual of $\mathcal{K}_{\bx}(e^t)$.
 Hence, by Minkowski's  second convex body theorem, it follows that
 \begin{eqnarray*}
   \log \lambda_1\left(\wedge^{d+1}\ZZ^{n+1},
   \mathcal{K}_{\bx}^{d+1}(\e^{t})\right) &=& \sum_{i=1}^{d+1} \log \lambda_i \left(\ZZ^{n+1},\mathcal{K}_{\bx}(\e^{t})\right)+O(1)\\
    &=& \sum_{i=1}^{d+1} -\log\lambda_{n+2-i} \left(\ZZ^{n+1},\mathcal{K}_{\bx}(e^t)^{\vee}\right)+O(1) \\
    &=& \sum_{i=1}^{n-d} \log \lambda_i  \left(\ZZ^{n+1},\mathcal{K}_{\bx}(e^t)^{\vee}\right) +O(1)\\
    &=& -\frac{(n-d)t}{n+1}+\sum_{i=1}^{n-d}L_{\bx,i}(t) +O(1)  \, ,
 \end{eqnarray*}
where the quantity $L_{\bx,i}(t)$ is given by \eqref{def-convexbodySV} and the implied constants  in the `big $O$' term depend on $n$ and $d$ only.  This thereby completes the proof of first claim made in the lemma. The proof of (ii) is similar and  we leave the details to the reader.
 \end{proof}

\medskip

\subsubsection{An optimal Dirichlet type theorem via multilinear algebra}  \label{DTTMA}

For obvious reasons, as discussed in Remark~\ref{bull}, in order to define Dirichet improvable sets  it is paramount to start with an optimal Dirichlet type theorem.  Any such theorem should naturally  not only imply Theorem~\ref{Dir-d} concerning the approximation of points by rational subspaces but also coincide with the classical simultaneous and dual forms of Dirichlet theorem.  The multilinear algebra framework exploited in the previous section yields the following optimal statement.

\begin{thm}\label{t-i-d}
Let $n \in \NN$ and $d $ be an integer satisfying $ 0 \le d \le n-1$. Then for any  $\bx\in \RR^n$ and $N > 1$,  there exist $\bZ \in \wedge^{d}\ZZ^{n}\setminus \{0\}$ and $\bY \in \wedge^{d+1}\ZZ^{n}$ such that
\begin{equation}\label{e-qqq}
\|\bZ\|  \le N \quad \text{ and } \quad \|\btheta \wedge \bZ+\bY\| \le  N^{-\omega_d} \, .
\end{equation}
\end{thm}

\medskip

\begin{proof}
Consider the linear space $V:=\wedge^d \RR^{n}\oplus \wedge^{d+1} \RR^{n}$ with the lattice $L:=\wedge^d \ZZ^{n}\oplus \wedge^{d+1} \ZZ^{n}$. On  observing that $\omega_d={n \choose d}/{n \choose d+1}$, it is easily verified that for any given $\bx\in \RR^{n}$ the volume of the convex body given by \eqref{e-qqq} is equal to  $2^{\dim V}$. Hence, the statement of the theorem follows as a direct consequence of Minkowski's convex body theorem.
\end{proof}

\medskip

\noindent   By taking $d=0$ and  $n-1$ in Theorem~\ref{t-i-d}, we immediately recover the classical simultaneous and dual forms of Dirichlet's theorem.  We now show that for general  $d$  we  recover  Theorem~\ref{Dir-d}.   \\

\noindent {\em Step 1.\ }   We show that Theorem~\ref{Dir-d} has the following equivalent algebraic formulation.  Recall, given $\vx \in \RR^n$ we  let $\bx':=(\bx, 1) \in \RR^{n+1} $.

\begin{lemma}\label{l-d-type}
Let $n \in \N$ and $d$ be integer satisfying $0 \le d \le n-1$.  Then
 for any $\vx \in \RR^n$  there exists a constant $ c=c(n,d,\bx) > 0$,  such that for any  $N \ge 1 $ there exist a $\bX \in \wedge^{d+1}\ZZ^{n+1}\setminus \{0\}$, such that
      \begin{equation}\label{e-equation-dir}
        |\bX|\le N \quad \text{and} \quad |\bx'\wedge \bX|\le c N^{-\omega_d}.
      \end{equation}
\end{lemma}

\begin{proof}[Proof of equivalence of Theorem~\ref{Dir-d} and Lemma \ref{l-d-type}.]     For the same reasons as outlined in Remark~\ref{corona}, it is easy to deduce Lemma~\ref{l-d-type} from Theorem~\ref{Dir-d}. For the converse, we adapt the proof of Lemma~\ref{l-bad-sing}.   This simply amounts to putting  $\eps = c(n,d,\bx)$  when defining the convex body $\mathcal{C}$ given by \eqref{slv-compound}. Apart from this the given  proof remains unchanged.
\end{proof}

\bigskip

\noindent {\em Step 2.\ }   We show that  Theorem~\ref{t-i-d} implies Lemma~\ref{l-d-type}.

\begin{proof}[Proof  that Theorem~\ref{t-i-d} implies Lemma \ref{l-d-type}.]
With reference to Theorem~\ref{t-i-d}, given $\vx \in \RR^n$ and  $N> \max\{|\bx|^{-1/\omega_d}, 1\}$ let $\bZ \in \wedge^{d}\ZZ^{n}\setminus \{0\}$ and $\bY \in \wedge^{d+1}\ZZ^{n}$ be a solution to  \eqref{e-qqq}. Then, it follows that
\begin{equation}\label{e-111sv}
\|\bY\|\le \|\bx\wedge \bZ\|+N^{-\omega_d} \le |\bx\wedge \bZ|+|\bx|\le 2|\bx||\bZ|.
\end{equation}

\noindent Now let $\be_{n+1}:=(0,\ldots, 0, 1)\in \bZ^{n+1}$ and identify the set  $\{\by=(y_1,\ldots,y_{n+1})\in \ZZ^{n+1}: y_{n+1}=0 \}$ (resp. $\{\by=(y_1,\ldots,y_{n+1})\in \RR^{n+1}: y_{n+1}=0 \}$) with $\ZZ^n$ (resp. $\RR^n$). Then,  we have that
\[\wedge^{d+1}\ZZ^{n+1}= \be_{n+1}\wedge (\wedge^{d}\ZZ^{n})\oplus \wedge^{d+1}\ZZ^{n}.\]
Next, let
\[\bX:=\be_{n+1}\wedge \bZ-\bY   \, \in  \,  \wedge^{d+1}\ZZ^{n+1}.\]
Then, it follows that
\begin{eqnarray*}
\|\bX\| \ \le \  \|\bZ\|+ \|\bY\| \ \stackrel{\eqref{e-111sv}}{\le} \  (2|\bx|+1)|\bZ|  & \le  &  (2|\bx|+1)\sqrt{{n \choose d}} \ \|\bZ\|  \nonumber \\[1ex] & \stackrel{\eqref{e-qqq}}{\le}  &  2^{\frac{n}{2}}(2|\bx|+1)N \ ,
\end{eqnarray*}
and so
\begin{equation} \label{e-111}
|\bX|   \, \le  \,   n^{\frac12} \, 2^{\frac{n}{2}}(2|\bx|+1)N \, .
\end{equation}  \\
On the other hand, we have that
\begin{align*}
  \bx'\wedge \bX &= (\be_{n+1}+\bx)\wedge (\be_{n+1} \wedge \bZ-\bY)=-\be_{n+1}\wedge\bY+ \bx\wedge \be_{n+1} \wedge \bZ -\bx\wedge \bY \\[1ex]
  &=-(\be_{n+1}+\bx)\wedge(\bx \wedge \bZ+\bY).
\end{align*}
Since $\be_{n+1}$ is orthogonal to $\bx$, $\be_{n+1}\wedge (\bx \wedge \bZ+\bY)$ is orthogonal to $\bx \wedge(\bx \wedge \bZ+\bY)$. Thus,
\begin{equation}\label{e-222sv}
  |\bx \wedge \bZ+\bY| \, \le  \,  |\bx'\wedge \bX| \, \le  \,  |\bx'| |\bx \wedge \bZ+\bY|
\end{equation}
and on using the above right hand side inequality and  \eqref{e-qqq},    it follows that
\begin{equation}\label{e-222}
|\bx'\wedge \bX|  \, \le  \,
2^{\frac{n}{2}}|\bx'|  \,  \|\bx \wedge \bZ+\bY\| \, \le  \,  2^{\frac{n}{2}}|\bx'|N^{-\omega_d} .
\end{equation}\\
Together,  \eqref{e-111} and \eqref{e-222} imply the statement of the lemma.
\end{proof}

The upshot of Steps 1 $\&$ 2 is the following desired statement.

\begin{prop} \label{propa1} \ \ \  \
 Theorem~\ref{t-i-d}  $ \quad \Longrightarrow \quad $    Theorem~\ref{Dir-d}.
\end{prop}

\subsubsection{Bad, Singular and Dirichlet Improvable in light of Theorem~\ref{t-i-d}.   \label{bojo}}

Theorem~\ref{t-i-d} enables us to define an associated  notion of $d$-Dirichlet improvable vectors as well as alternative ``algebraic'' notions of $d$-singular and $d$-badly approximable vectors.

\begin{definition} \label{BLdef}
 Let $n \in \NN$ and $d $ be an integer satisfying $ 0 \le d \le n-1$.  Then for $\eps \in (0,1)$, let
 \begin{itemize}
   \item  $\DI_n^d(\eps)$  be the set of $\bx\in \RR^n$, such that for  all $N > N_0( \bx, \eps) $ sufficiently large,  there exist $\bZ \in \wedge^{d}\ZZ^{n}\setminus \{0\}$ and $\bY \in \wedge^{d+1}\ZZ^{n}$  such that
\begin{equation}\label{e-equation-sing2}
\|\bZ\|  \le N  \quad \text{ and } \quad \|\btheta \wedge \bZ+\bY\| \le \eps N^{-\omega_d}.
\end{equation}
\item $\Bad_n^d(\eps)$ be the set of $\bx\in \RR^n$, such that for  all $N > N_0( \bx, \eps) $ sufficiently large,  there are no solutions  $\bZ \in \wedge^{d}\ZZ^{n}\setminus \{0\}$ and $\bY \in \wedge^{d+1}\ZZ^{n}$ to  \eqref{e-equation-sing2}.
 \end{itemize}
Moreover, let
\[\DI_n^d \, :=\bigcup_{\eps\in (0,1)} \DI_n^d(\eps),\quad \Sing_n^d \, := \bigcap_{\eps\in (0,1)} \DI_n^d(\eps),\quad \Bad_n^d \, := \bigcup_{\eps\in (0,1)} \Bad_n^d(\eps).
 \]
\end{definition}

\medskip

{Recall, that on taking $d=0$ (resp. $d=n-1$) in  Theorem~\ref{t-i-d}, we immediately recover the classical simultaneous (resp. dual) form of Dirichlet's theorem. Thus, in these cases it is clear that the sets $\DI_n^d$, $\Sing_n^d$  and $\Bad_n^d$ given by Definition \ref{BLdef} coincide with those defined via the  classical forms of Dirichlet theorem.
In particular, for Dirichlet improvable points in $\R^n$, this means that the set $ \DI_n^0 $  (resp.  $ \DI_n^{n-1} $)  is  equivalent to  set  defined via the classical inequality \eqref{DIsv1} (resp.  \eqref{equ-basic}) and
in view of Davenport $\&$ Schmidt  \cite[Theorem~2]{DavSch} we have that
 \begin{equation} \label{stillagain}
 \DI_n^0=\DI_n^{n-1}:=\DI_n  \, ;
 \end{equation}
 i.e. the simultaneous  form of Dirichlet's theorem is improvable if and only if the dual form of Dirichlet's theorem is improvable.
}

For obvious reasons, including consistency, it is vitally important that the following `equivalence' statement is true for previously defined intermediate Diophantine sets.

\begin{lemma} \label{age}
The definition of $\Sing_n^d$  (resp. $\Bad_n^d$)  given by Definition \ref{BLdef} is equivalent to that given by \eqref{equ-gen}
  (resp. \eqref{slv-bad-d}) in  \S\ref{setupmain}.
\end{lemma}

\begin{proof}
Let us write  $\widetilde{\Sing}_n^d$ (resp. $\widetilde{\Bad}_n^d$) to denote the set of $d$-singular vectors  (resp. $d$-badly approximable vectors )  given by \eqref{equ-gen} (resp. \eqref{slv-bad-d}).  Then,   Lemma~\ref{l-bad-sing} states  that
\begin{itemize}
\item $\bx \in \widetilde{\Sing}_n^d$ if and only if for any given $\eps \in (0,1)$ and $N >N_0(\bx, \eps)$ sufficiently large, there exists $\bX\in \wedge^{d+1} \ZZ^{n+1}\setminus\{0\}$, such that
      \begin{equation}\label{e-equation-singA}
        |\bX|\le N     \quad \text{and} \quad |\bx'\wedge \bX|\le \eps N^{-\omega_d}.
      \end{equation}
  \item  $\bx\in \widetilde{\Bad}_n^d$ if and only if there exists a constant $\eps:=\eps(\bx) \in (0,1)$
      such that for any $N >N_0(\bx)$ sufficiently large, there are no solutions $\bX\in \wedge^{d+1} \ZZ^{n+1}\setminus\{0\}$ to \eqref{e-equation-singA}.
\end{itemize}
As in the proof of Step~2 in \S\ref{DTTMA},   let $\be_{n+1}=(0,\ldots, 0, 1)\in \bZ^{n+1}$ and identify the set  $\{\by=(y_1,\ldots,y_{n+1})\in \ZZ^{n+1}: y_{n+1}=0 \}$ (resp. $\{\by=(y_1,\ldots,y_{n+1})\in \RR^{n+1}: y_{n+1}=0 \}$) with $\ZZ^n$ (resp. $\RR^n$). Then,  we have that
\[\wedge^{d+1}\ZZ^{n+1}= \be_{n+1}\wedge (\wedge^{d}\ZZ^{n})\oplus \wedge^{d+1}\ZZ^{n} \, .\]
Hence, we can write any $\bX\in \wedge^{d+1}\ZZ^{n+1}\setminus \{0\}$ uniquely as $\be_{n+1}\wedge \bZ-\bY$ with $\bZ \in \wedge^{d}\ZZ^{n}$ and $\bY \in \wedge^{d+1}\ZZ^{n}$ such that either $\bZ$ or $\bY$ is non-zero. Then, the same argument leading to \eqref{e-222sv}, enables us to conclude that
\begin{equation}\label{e-compare1}
|\bx \wedge \bZ+\bY|\le |\bx'\wedge \bX| \le |\bx'| |\bx \wedge \bZ+\bY|.
\end{equation}
Note that we also have that
\begin{equation}\label{e-compare2}
 \max\{|\bZ|, |\bY|\}\le |\bX| \le 2\max\{|\bZ|, |\bY|\}.
\end{equation}

Now fix  $\eps\in (0,1)$ and $  \bx  \in \RR^n$.  Let
$N> \max\{|\bx|^{-1/\omega_d}, 1\}$ and suppose that  $\bZ \in \wedge^{d}\ZZ^{n}\setminus \{0\}$ and $\bY \in \wedge^{d+1}\ZZ^{n}$ is a solution to \eqref{e-equation-sing2}.   Let $\bX\in \wedge^{d+1}\ZZ^{n+1}$ be the corresponding unique point satisfying \eqref{e-compare1} and \eqref{e-compare2}.   Then, the same argument leading to   \eqref{e-111} and \eqref{e-222},  shows that $\bX$ satisfies
\[|\bX|   \, \le  \,   n^{\frac12} \, 2^{\frac{n}{2}}(2|\bx|+1)N \quad \text{ and }\quad
 |\bx'\wedge \bX|\le \eps 2^{\frac{n}{2}}|\bx'|N^{-\omega_d}. \]
 Conversely, let $N> \max\{|\bx|^{-1/\omega_d}, 1\}$ and suppose that $\bX \in \wedge^{d+1}\ZZ^{n+1}\setminus \{0\}$  is a solution to \eqref{e-equation-singA}. Let $\bZ \in \wedge^d \ZZ^n$ and $\bY \in \wedge^{d+1} \ZZ^n$ be the corresponding unique points (not both zero)  satisfying \eqref{e-compare1} and \eqref{e-compare2}.  Then it follows that using  the
 left hand side of \eqref{e-compare1} and \eqref{e-compare2}, that  $\bZ $ and $\bY $ satisfy
\[ \|\bZ\|  \le N  \quad \text{ and } \quad \|\btheta \wedge \bZ+\bY\| \le \eps N^{-\omega_d}.\]
It remains to show that $\bZ \neq 0 $.   Suppose it is. Then since the right hand side of the second inequality is strictly less than one, we must have that  $\bY=0$.   This contradicts the assumption that both are not zero.  Hence we must have $\bZ \in \wedge^d \ZZ^n\setminus \{0\}$. This completes the proof.
\end{proof}

\bigskip

We now proceed by describing a  dynamical reformulation of the  intermediate Diophantine sets associated with Definition~\ref{BLdef}.  This will  enable us to provide an alternative proof of Proposition \ref{prop-equiv} that is self-contained in that it avoids appealing to \cite[Lemma 3]{BugLau2}.    Moreover, the dynamical reformulation  will enable us to extend  the  Davenport $\&$ Schmidt result \cite[Theorem 2]{DavSch} concerning the equivalence of the simultaneous and dual Dirichlet improvable sets to intermediate Dirichlet improvable sets.

For simplicity,  given  $n\ge 2$ and $0\le d\le n-1$, we write
\[A_d:={n \choose d} \; ,  \quad  B_d:={n \choose d+1} \quad \text{ and } \quad  N_d:=A_d+B_d={n+1 \choose d+1}.\]
For $\bx\in \RR^n$, let $  H_{d,\bx}$  to be the linear transformation defined as
\[H_{d,\bx}: \wedge^d \RR^{n} \rightarrow \wedge^{d+1} \RR^{n} \, : \,  \bZ \mapsto \bx\wedge \bZ.\]
Under the standard basis, $H_{d,\bx}$ is given as a matrix $M_{d, \bx}\in M_{A_d\times B_d}(\RR)$.
Recall,  a matrix $M\in M_{A\times B}(\RR) $ of $A$ rows and $B$ columns, is said to be
\begin{itemize}
  \item \emph{Dirichlet improvable} if and only if  there exists $\eps=\eps(M)\in (0, 1)$ such that for all $N\ge N_0(M, \eps)$ sufficiently large, there exists $\bY\in \ZZ^A$ and $\bZ\in\ZZ^B\setminus \{0\}$ such that
      \begin{equation}\label{def-matrix}
        \|\bZ\|\le N \quad \text{ and } \quad \|M\bZ+\bY\|\le \eps  N^{-\frac{B}{A}}  \, .
      \end{equation}
  \item \emph{singular} if and only if for any given $\eps\in (0,1)$ and  $N\ge N_0(M, \eps)$ sufficiently large, there exists $\bY\in \ZZ^A$ and $\bZ\in\ZZ^B\setminus \{0\}$ such that \eqref{def-matrix} holds.
  \item \emph{badly approximable} if and only if there exists $\eps=\eps(M)\in (0, 1)$ such that for all $N\ge N_0(M, \eps)$ sufficiently large, there are no solutions $\bY\in \ZZ^A$ and $\bZ\in\ZZ^B\setminus \{0\}$ to \eqref{def-matrix}.
\end{itemize}

\medskip

\noindent On comparing the notions as given in Definition \ref{BLdef} with the corresponding matrix notions above we immediately obtain the following statement.

\begin{lemma}\label{l-easy}  Let $ \bx \in \R^n$, $0\le d\le n-1$ be an integer, $A=A_d$ and $B=B_d$. Then
$\bx$ is $d$-Dirichlet improvable (resp. $d$-badly approximable,  $d$-singular)  if and only if the matrix $M=M_{d, \bx}$ is Dirichlet improvable  (resp. badly approximable, singular).
\end{lemma}

Using this lemma, we can provide another proof of Proposition \ref{prop-equiv}.
\begin{proof}[Alternative proof of Proposition \ref{prop-equiv}.]
 Let
 \begin{equation} \label{qslv3}
\rho_d:\SL_{n+1}(\RR) \rightarrow \SL_{N_d}(\RR)  \, ,
\end{equation}
be the homomorphism induced by the natural representation on $\wedge^{d+1} \RR^{n+1}$. It is easily seen that
\[\rho_d(\SL_{n+1}(\ZZ))\subset \SL_{N_d}(\ZZ) \quad \text{ and } \quad \rho_d(g_t)=g_t^d  \, , \]
where
\[g_t:=\begin{pmatrix} e^{-t} I_n & 0 \\0 & e^{nt} \end{pmatrix} \quad \text{ and } \quad g_t^d:= \begin{pmatrix}
e^{-B_d t}I_{A_d} & 0 \\ 0 & e^{A_d t}I_{B_d} \end{pmatrix}. \]
According to \cite[Theorem 1.13]{Rag}, the induced map
\begin{equation} \label{wslv3} \phi_d: X_{n+1}\rightarrow X_{N_d} \quad \text{ where } \quad X_m:=\SL_m(\RR)/\SL_m(\ZZ)
\end{equation}
 is proper. This together with Dani's correspondence \cite{Dani} and Lemma~\ref{l-easy}, implies that
 \begin{align*}
   \text{$\bx$ is $d$-badly approximable} &\Longleftrightarrow \text{ the orbit $\{g_t^d \Lambda_{M_{d, \bx}}: t\ge 0\}$ is bounded} \\
   &\Longleftrightarrow \text{ the orbit $\{g_t \Lambda_\bx: t\ge 0\}$ is bounded} \\
   &\Longleftrightarrow \text{ $\bx$ is badly approximable} \, .
 \end{align*}
  Here and elsewhere,
 \[\Lambda_{\bx}:=\begin{pmatrix}  I_n & 0 \\ \bx & 1 \end{pmatrix}\ZZ^{n+1} \quad \text{ and } \quad \Lambda_{M_{d, \bx}}:= \begin{pmatrix}  I_{A_d} & 0 \\ M_{d, \bx} & I_{B_d} \end{pmatrix}\ZZ^{N_d}.\]
 The proof of equivalence in the singular case  is similar and  we leave the details to the reader.
\end{proof}

Concerning the set of intermediate Dirichlet improvable vectors, we are able to prove the following statement. {It shows that the sets $\DI_n^d$ are all the same. This is of course  in perfect  harmony  with the situation concerning the intermediate badly approximable and singular sets.  }

\begin{prop}\label{prop-d-di}
Let $n \in \NN$ and $d$ be an integer satisfying  $0\le d\le n-1$. Then $$\DI_n^d=\DI_n. $$
\end{prop}

\noindent {Recall, that  $\DI_n$ is the set of Dirichlet improvable points in $\R^n$ defined via  the classical simultaneous or (equivalently) dual form of  Dirichlet's theorem.  Note that the proposition trivially implies the Davenport $\&$ Schmidt result \cite[Theorem 2]{DavSch} conveniently  summed up by \eqref{stillagain}. }
 The proof of  Proposition~\ref{prop-d-di} is
 based on a dynamical reformulation of $\DI_n^d$, which in turn relies on the following theorem of Haj{\'o}s \cite{Hajos}.

\begin{thm}[Haj{\'o}s]\label{Hajos}
Let $k\in \NN$ and $L\subset \RR^k$ be a lattice of covolume $1$. Then
\[ L \cap \Pi_k= \emptyset \Longleftrightarrow L \in \bigcup_{w\in \Sym_k} w^{-1}U_k w \ZZ^k,
\]
where $\Pi_k:=(-1,1)^k$ denotes the open unit cube in $\RR^k$, $\Sym_k$ represents the Weyl group and $U_k\subset \SL_k(\RR)$ denotes the subgroup  of upper triangular matrices with all diagonal entries equal to $1$.
\end{thm}

{By Dani's correspondence, $\bx\in\DI_n$  if and only if }the $\omega$-limit set
\[\left\{\Lambda\in X_{n+1}: \text{ there exists } (t_k)_{k\ge 0} \text{ with } \lim_{k\rightarrow \infty}t_k=\infty \text{ and } g_{t_k}\Lambda_\bx=\Lambda \right\} \]
does not intersect the set
\[E:=\{\Lambda\in X_{n+1}: \Lambda\cap \Pi_{n+1}=\emptyset\}.\]
It follows from Theorem \ref{Hajos}, that
\[E=\bigcup_{w\in \Sym_{n+1}} w^{-1}U w \ZZ^{n+1}.\]
Let $ \phi_d$ be the map given by \eqref{wslv3},
then, in view of Lemma \ref{l-easy}, $\bx$ is $d$-Dirichlet improvable if and only if the $\omega$-limit set
\[\left\{\Lambda\in X_{N_d}: \text{ there exists } (t_k)_{k\ge 0} \text{ with } \lim_{k\rightarrow \infty}t_k=\infty \text{ and } \phi_d(g_{t_k}\Lambda_\bx)=\Lambda \right\} \]
does not intersect the set
\[E_{d}:=\bigcup_{w\in \Sym_{N_d}} w^{-1}U_{N_d} w \ZZ^{N_d}.\]
Thus, $\bx$ is $d$-Dirichlet improvable if and only if the $\omega$-limit set
\[\left\{\Lambda\in X_{n+1}: \text{ there exists } (t_k)_{k\ge 0} \text{ with } \lim_{k\rightarrow \infty}t_k=\infty \text{ and } g_{t_k}\Lambda_\bx=\Lambda \right\} \]
does not intersect the set
\[E_{d}':=X_{n+1}\cap \phi_d^{-1}(E_d).\]

The upshot of the above is that  Proposition \ref{prop-d-di} is {a direct consequence of} the following statement.

\begin{lemma}\label{lem-e-ed}
Let $n \in \NN$ and $d$ be an integer satisfying  $0\le d\le n-1$.   Then, $E=E_d'$.
\end{lemma}

We now proceed with the proof of Lemma~\ref{lem-e-ed}.  This will be done in two steps.

\medskip

\noindent{$\bullet$} {\em Step 1:}  $E\subseteq E_d' \, . \ $ This inclusion is straightforward.
Indeed, for any $w\in \Sym_{n+1}$, we have that
\[\rho_d(w^{-1}U_{n+1}w)  \subseteq \psi_d(w)^{-1}U_{N_d}\psi_d(w),\]
where $ \rho_d$ is given by \eqref{qslv3} and  $\psi_d$ is the natural map from $\Sym_{n+1}$ to $\Sym_{N_d}$.
Hence, it follows that  $E \subseteq  E_d'$.

\medskip

\noindent{$\bullet$} {\em Step 2:}  $E\supseteq E_d' \, . \ $
To prove this inclusion, it suffices to show that for any $w\in \Sym_{N_d}$, there exists $w_1\in \Sym_{n+1}$ such that
\begin{equation}\label{equ-crucial}
X_{n+1}\cap \phi_d^{-1}\left( w^{-1}U_{N_d} w \ZZ^{N_d} \right) \subseteq w_1^{-1}U w_1\ZZ^{n+1}.
\end{equation}
Without loss of generality, let us fix $w\in \Sym_{N_d}$ and write
\begin{equation}\label{vb-3}
H_1:=\phi_d(\SL_{n+1}(\RR)) \quad \text{ and } \quad H_2:= w^{-1}U_{N_d} w.
\end{equation}
The proof of the inclusion \eqref{equ-crucial} will make use of the  following technical lemma.

\begin{lemma}\label{lem-tech}
  Let $Y=H_1H_2\subset \SL_{N_d}(\RR)$. Then  $Y$ is a variety defined over $\ZZ$ and
\[Y(\ZZ)=  H_1(\ZZ) H_2(\ZZ).\]
\end{lemma}

\noindent For the moment, let us assume the truth of the lemma and continue with the proof of Step~2.  With this in mind, let $\Gamma=\SL_{N_d}(\ZZ)$ and note that
\[H_1\Gamma \cap H_2\Gamma=\bigcup_{\gamma\in \Gamma} (H_1 \gamma \cap H_2) \Gamma.\]
Now if
$H_1 \gamma \cap H_2\ne \emptyset$, then $\gamma\in Y(\ZZ)$ and in view of Lemma \ref{lem-tech}, it follows that  $\gamma\in H_1(\ZZ) H_2(\ZZ)$. Thus,
\[(H_1 \gamma \cap H_2) \Gamma= (H_1\cap H_2) \Gamma.\]
and so
\begin{equation}\label{equ-claim-int-h12}
H_1\Gamma \cap H_2\Gamma=(H_1\cap H_2) \Gamma.
\end{equation}
Next, let $T$ be the diagonal subgroup of $\SL_{n+1}(\RR)$ with positive entries.  Then $D:=\phi_d(T)$ is also a  connected  diagonal subgroup of $\SL_{N_d}(\RR)$. Note that both $H_1$ and $H_2$ are invariant under the conjugation of $D$  and so it follows that  their intersection $H_1\cap H_2$ is also invariant under the conjugation of $D$. Note that the group $H_1\cap H_2$ is clearly unipotent, thus connected.  Hence $D(H_1\cap H_2)$ is a connected solvable subgroup of $H_1$. Then there exists a Borel subgroup $B$ such that $D(H_1\cap H_2) \subset B$. Thus, there exists $w_1\in \Sym_{n+1}$ such that
\begin{equation}\label{equ-claim-int-h}
\phi_d^{-1}(H_1\cap H_2) \subset w_1^{-1} U w_1.
\end{equation}
Now \eqref{equ-crucial} follows from \eqref{equ-claim-int-h12} and \eqref{equ-claim-int-h}. This completes the proof of Step 2 and thereby the  proof of Lemma~\ref{lem-e-ed} modulo Lemma~\ref{lem-tech}.
\vspace*{-2ex}

\hfill $\square$

\medskip
 \begin{remark} \label{maybe}
With $n$ and $d$ as in Lemma \ref{lem-e-ed}, it is relatively straightforward to prove the weaker statement that
 \begin{equation} \label{easybit} \DI_n^d=\DI_n^{n-1-d}
 \end{equation}
without appealing to Lemma~\ref{lem-tech}. Note that when $d=0$ or $n-1$, this implies the Davenport $\&$ Schmidt result summed up by \eqref{stillagain}.  To prove  \eqref{maybe}, let $\tau_d: X_{N_d} \rightarrow X_{N_{n-d-1}}$ be the isomorphism induced by the Hodge star operator, which is a natural isomorphism between $\wedge^{d+1}(\RR^{n+1})$ and $\wedge^{n-d}(\RR^{n+1})$, see \S\ref{Sec_A2.1} for its definition. It is easily verified  that $\tau_d(E_d)=E_{n-d-1}$ and $\tau_d\circ\phi_d=\phi_{n-d-1}$. Hence, it follows that $E_d'=E_{n-d-1}'$ which in turn this implies  the desired result.
\end{remark}

\medskip

The proof of Lemma \ref{lem-tech} will make use of tools from non-abelian cohomology.  For a general reference on non-abelian cohomology, we refer the reader to Giraud's book \cite{Gir}.  On considering the action of $H:=H_1\times H_2$ on $Y$ defined by $(h_1, h_2) \times y \mapsto \psi_d(h_1)yh_2^{-1}$  we can identify $Y$ with the quotient space $H/(H_1\cap H_2)$.   It follows that
Lemma~\ref{lem-tech} is equivalent to saying  that $Y(\ZZ)$ consists of a single orbit under the action of $H(\ZZ)$. The space of orbits $Y(\ZZ)/H(\ZZ)$ is controlled by the non-abelian cohomology of $\ZZ$ with coefficients in $H_1\cap H_2$. The non-abelian cohomology that we are going to use is the fppf-cohomology and we will always work over $\Spec (\ZZ)$ - the spectrum of the commutative ring $\ZZ$ of integers. For any group scheme $\bH$ defined over $\Spec (\ZZ)$, we let
\[H^1(\ZZ, \bH):= H^1_{\mathrm{fppf}}\left(\Spec (\ZZ), \bH \right)\]
denote the fppf-cohomology of $\Spec (\ZZ)$ with coefficients in $\bH$.  With this in mind,  are now in the position to prove the lemma.

\begin{proof}[Proof of Lemma \ref{lem-tech}.]
For the sake of simplicity, we let $\bG:=\SL_{N_d}$ and  $\bH_1:=\SL_{n+1}$.  Let $\bT\subset \bH_1$ be the subgroup scheme of diagonal matrices. Let $\bH_2$ be the group scheme over $\ZZ$ given by $H_2$ endowed with their $\ZZ$-structure, where $H_2$ is given by \eqref{vb-3}.
 Let $\bY$ be the scheme defined over $\ZZ$ given by $Y$ endowed with its $\ZZ$-structure, where $Y$ is given in Lemma~\ref{lem-tech}. It is clear that, all these schemes are flat (indeed smooth) over $\ZZ$.

 There is an action of $\bH:=\bH_1\times \bH_2$ on $\bY$ such that for any ring $R$, we have
\[\bH(R) \times \bY(R) \to \bY(R), \quad (h_1, h_2)\times y \mapsto \phi_d(h_1) y h_2^{-1}.\]
Let $\bH_0=\bH_1\times_{\bG} \bH_2$; that is, for any ring $R$
\[
\bH_0(R)=\bH_1(R)\times_{\bG(R)} \bH_2(R)\,.
\]
It follows that $\bH_0$ is a  unipotent subgroup scheme of $\bH_1$ smooth over $\ZZ$ and is normalized by $\bT$. Consequently, $\bH_0$ is split over $\ZZ$. In particular, there is a sequence of smooth unipotent $\ZZ$-group schemes
\[\{1\}=\bH_0^{\dim \bH_0}\subset \bH_0^{\dim \bH_0-1}\subset \cdots \subset \bH_0^{2}\subset \bH_0^{1}=\bH_0\]
such that $\bH_{0}^{i+1}$ is normal in $\bH_{0}^{i}$ and $\bH_{0}^{i}/\bH_{0}^{i+1}\cong \bG_a$.
Now observe that $H^1(\ZZ, \bG_a)=\{1\}$. Then on arguing by induction on the dimension and making use of   \cite[Proposition~3.3.1]{Gir}, it follows that
\begin{equation}\label{equ-fppf}
H^1(\ZZ, \bH_0)=\{1\}.
\end{equation}
Furthermore, it is easily checked that $\bY$ is isomorphism to the quotient $\bH/\bH_0$. In view of \cite[Corollary 3.2.3]{Gir}, there is an injection
\[\bY(\ZZ)/\bH(\ZZ) \hookrightarrow H^1(\ZZ, \bH_0),\]
where $\bY(\ZZ)/\bH(\ZZ)$ denotes the orbit space of $\bY(\ZZ)$ under the action of $\bH(\ZZ)$. It follows from \eqref{equ-fppf} that $\bY(\ZZ)$ consists of a single orbit under the action of  $\bH(\ZZ)$. This completes the proof of Lemma \ref{lem-tech}.
\end{proof}

{

\subsection{The geometric approach}  \label{DTTGN}

In this section we explore a geometric approach to defining $d$-Dirichlet improvable sets. The approach taken is in line with that of \S\ref{sec-reform} in which the $d$-badly approximable and $d$-singular sets are expressed via successive minima; namely via \eqref{equ-equi} and \eqref{equ-equiSING}.

\subsubsection{An optimal Dirichlet type theorem via successive minima}\label{Sec_A2.1}


Fix $ n \in \N$  and $ \bx \in \R^n$.    Then, as in \S\ref{PGN},  for each $i=1, \ldots, n+1$ and $t>0$, let
$\lambda_{\bx,i}(t):=\lambda_i\big(\ZZ^{n+1}, \sC_{\btheta}(\e^t)\big)$
be  the $i$-th successive minima of the convex body $\sC_{\btheta}(\e^t)$ given by \eqref{def-sc} with respect to the lattice $\Z^{n+1}$. In turn,  let
\begin{equation}\label{vb301}
L_{\bx, i}(t):=\log\lambda_{\bx,i}(t)
\end{equation}
be the natural logarithm of $\lambda_{\bx,i}(t)$.   It follows, via Minkowski's second convex body theorem,  that
\begin{equation}\label{vb8A}
\frac{e^t}{(n+1)!}  \, \le  \, \prod_{i=1}^{n+1}\lambda_{\bx,i}(t) \, \le  \,  e^t
\end{equation}
or equivalently
$$
t-\log(n+1)! \, \le  \, \sum_{i=1}^{n+1}L_{\bx,i}(t)  \, \le  \,  t\,.
$$
Now,  \eqref{vb8A} together with the fact that  $\lambda_{\bx,1}(t)\le \lambda_{\bx,2}(t) \le\dots\le\lambda_{\bx,n+1}(t)$ gives rise to  the following statement.

\begin{thm}\label{DirInterm}
Let $n \in \NN$ and $d $ be an integer satisfying $ 0 \le d \le n-1$. Then for any  $\bx\in \RR^n$ and $t > 0$
\begin{equation}\label{vb101B}
\prod_{i=1}^{n-d}\lambda_{\bx, i} \big(t) \; \le  \; e^{\frac{(n-d)t}{n+1}}\, ,
\end{equation}
or equivalently
\begin{equation*}\label{vb101C}
\sum_{i=1}^{n-d}L_{\bx, i}(t)\le \frac{(n-d)t}{n+1}\,.
\end{equation*}
\end{thm}

Theorem~\ref{DirInterm} can be viewed as an optimal Dirichlet type theorem for approximating points $\bx \in \R^n$ by $d$-dimensional rational subspaces. In the dual case ($d=n-1$) when we approximate by hyperplanes,  this realisation is relatively straightforward to see. For arbitrary $d\in\{0,\dots,n-1\}$, the following statement provides the link.
Recall that, by the definition of successive minima, given $\bx$ and $t>0$, there exist linearly independent vectors $\bv_1,\dots,\bv_{n+1}\in\Z^{n+1}$, depending on $\bx$ and $t$, such that
\begin{equation}\label{vb102}
\bv_i  \, \in  \, \lambda_{\bx,i}(t)  \; \sC_{\bx}(\e^t) \qquad (1\le i\le n+1)\,.
\end{equation}
As we shall soon see, the vectors play a key role in detecting the sought after  $d$-dimensional  rational subspaces. As usual, given  $ \bx \in \R^n$, we let   $\bx':=(\bx, 1)\in \RR^{n+1}$  and  we also let
$$\|\bx\|_1  \, :=  \, |x_1|+\dots+|x_n|   \, . $$

\begin{lemma}\label{lemB2A}
Let $n \in \NN$ and $d $ be an integer satisfying $ 0 \le d \le n-1$. Then for any  $\bx\in \RR^n$ and $t > 0$, let $\bv_1,\dots,\bv_{n+1}\in\Z^{n+1}$ be as in \eqref{vb102}.  In turn, let  $\bu_1,\dots,\bu_{d+1}$ be any basis of $V(\bv_1,\dots,\bv_{n-d})^\perp\cap\Z^{n+1}$, where $V(\bv_1,\dots,\bv_{n-d})^\perp$ is the linear subspace of $\R^{n+1}$ orthogonal to the vectors $\bv_1,\dots,\bv_{n-d}$. Finally, let  $\bX:=\bu_1\wedge\cdots\wedge\bu_{d+1}\in\bigwedge^{d+1}\Z^{n+1}$. Then
\begin{equation}
\|\bx'\wedge\bX\| \le (n-d)!(1+\|\bx\|_1)\left(\prod_{i=1}^{n-d}\lambda_{\bx,i}(t) \right) e^{-t}\label{vb110A}
\end{equation}
and
\begin{equation}
0<\|\bX\| \le (n-d)!(1+\|\bx\|_1)\left(\prod_{i=1}^{n-d}\lambda_{\bx,i}(t) \right)\label{vb110B}\,  .
\end{equation}
\end{lemma}

To prove the lemma, it will be useful to make use of the notion of the Hodge dual - the image of a multivector under the Hodge star operator. In relation to multivectors over $\R^{n+1}$, the Hodge star operator relates any basis $\ba_1,\dots,\ba_\ell$ of a vector subspace of $\R^{n+1}$ to a basis $\mathbf{b}_1,\dots,\mathbf{b}_{n+1-\ell}$ of the orthogonal subspace.  Naturally, we will use the symbol ${}^\perp$ to denote the Hodge dual within the context of multivectors over $\R^{n+1}$. In particular, $$\ba_1\wedge\dots\wedge\ba_\ell = \pm \,( \mathbf{b}_1\wedge\dots\wedge\mathbf{b}_{n+1-\ell})^\perp$$
provided the basis vectors under consideration  span parallelepipeds of the same volume.
We refer the reader to \cite[\S3]{Ber12} for  the definition of the Hodge star operator and its useful properties that we shall exploit. In particular, we have that for any indices $1\le i_1<\dots<i_{\ell}\le n+1$
\begin{equation}\label{hodge_coord}
(\be_{i_1}\wedge\dots\wedge \be_{i_\ell})^\perp=\pm \be_{j_1}\wedge\dots\wedge \be_{j_{n+1-\ell}}\,,
\end{equation}
where $\{j_1,\dots,j_{n+1-\ell}\}=\{1,\dots,n+1\}\setminus \{i_1,\dots,i_\ell\}$, and so the Hodge star operator preserves both Euclidean and maximum norms (in the standard basis).

\begin{proof}
To start with, simply note that  $\bX=\pm(\bv_1\wedge\cdots\wedge \bv_{n-d})^\perp\in\bigwedge^{d+1}\Z^{n+1}$ and so $\bX$ is   decomposable and non-zero. Without loss of generality,  we assume that the sign is positive.  By the properties of the   Hodge star operator, see \cite[Eq~(3.8)\&(3.9)]{Ber12}, we have that
\begin{equation}\label{vb204}
(\bX\wedge \bx')^\perp=((\bv_1\wedge\cdots\wedge \bv_{n-d})^\perp\wedge \bx')^\perp=
(-1)^{(d+1)(n-d)}(\bv_1\wedge\cdots\wedge \bv_{n-d})\cdot\vx'\,,
\end{equation}
where $(\bv_1\wedge\cdots\wedge \bv_{n-d})\cdot\vx'\in \bigwedge^{n-d-1}\R^{n+1}$ is the interior product of multivectors under consideration. By \cite[Eq~(3.6)]{Ber12}, the $\be_{j_1}\wedge\cdots\wedge\be_{j_{n-d-1}}$ coordinate of $(\bv_1\wedge\cdots\wedge \bv_{n-d})\cdot\vx'$ is equal to
$$
\big((\bv_1\wedge\cdots\wedge \bv_{n-d})\cdot\vx'\big)\cdot(\be_{j_1}\wedge\cdots\wedge\be_{j_{n-d-1}})=
(\bv_1\wedge\cdots\wedge \bv_{n-d})\cdot(\vx'\wedge\be_{j_1}\wedge\cdots\wedge\be_{j_{n-d-1}})\,.
$$
By Laplace's identity, see \cite[Eq~(3.3)]{Ber12}, this coordinate is equal to
\begin{equation}\label{vb103}
\det\left(\begin{array}{cccc}
             \bv_1\cdot\bx' & \bv_1\cdot\be_{j_1} & \dots & \bv_1\cdot\be_{j_{n-d-1}} \\
             \vdots & \vdots & \ddots & \vdots \\
             \bv_{n-d}\cdot\bx' & \bv_{n-d}\cdot\be_{j_1} & \dots & \bv_{n-d}\cdot\be_{j_{n-d-1}} \\
           \end{array}
\right)\,.
\end{equation}
By \eqref{vb102}, we have that $|\bv_i\cdot\bx'|\le \lambda_{\bx,i}(t) e^{-t}$ and $|\bv_i\cdot\be_{j}|\le \lambda_{\bx,i}(t)$ for $1\le j\le n$. In turn, since $\be_{n+1}=\bx'-\sum_{j=1}^{n}x_{j}\be_j$, we have that
$$
|\bv_i\cdot\be_{n+1}|=\left|\bv_i\cdot\bx'-\sum_{j=1}^{n}x_{j}\bv_i\cdot\be_j\right|\le \lambda_{\bx,i}(t) e^{-t}+\lambda_{\bx,i}(t) \|\bx\|_1
\le \lambda_{\bx,i}(t)\,(1+\|\bx\|_1)\,.
$$
Hence, by \eqref{vb103}, we get that
\begin{align*}
\|(\bv_1\wedge\cdots\wedge \bv_{n-d})\cdot\vx'\|& \le (n-d)!(1+\|\bx\|_1)\left(\prod_{i=1}^{n-d}\lambda_{\bx,i}(t) \right) e^{-t}\,.
\end{align*}
By \eqref{hodge_coord}, up to sign and order, the Hodge dual of a multivector has the same coordinates in the standard basis. Therefore, by \eqref{vb204} it follows that
\begin{align*}
\|\bx'\wedge\bX\|&=\|(\bv_1\wedge\cdots\wedge \bv_{n-d})\cdot\vx'\|\le (n-d)!(1+\|\bx\|_1)\left(\prod_{i=1}^{n-d}\lambda_{\bx,i}(t) \right) e^{-t}\, .
\end{align*}
This establishes \eqref{vb110A}.  Next,  calculating the $\be_{j_1}\wedge\cdots\wedge\be_{j_{n-d}}$ coordinate of $\bv_1\wedge\cdots\wedge\bv_{n-d}$ gives
\begin{equation}\label{vb8E}
(\bv_1\wedge\cdots\wedge \bv_{n-d})\cdot(\be_{j_1}\wedge\cdots\wedge\be_{j_{n-d}})\,.
=\det\left(\begin{array}{ccc}
             \bv_1\cdot\be_{j_1} & \dots & \bv_1\cdot\be_{j_{n-d}} \\
             \vdots & \ddots & \vdots \\
             \bv_{n-d}\cdot\be_{j_1} & \dots & \bv_{n-d}\cdot\be_{j_{n-d}} \\
           \end{array}
\right)\,.
\end{equation}
By the already established bounds on $|\bv_i\cdot\be_j|$, the absolute value of \eqref{vb8E} is bounded by
$$
(n-d)!(1+\|\bx\|_1)\left(\prod_{i=1}^{n-d}\lambda_{\bx,i}(t) \right)\,.
$$
This implies \eqref{vb110B} and thereby completes the proof of the lemma.
\end{proof}

\medskip

Given $\bx  \in \R^n$, $t> 0$ and  $i\in\{1,\dots,n+1\}$, it will be convinent to  define the following `normalised' value of the $i$-th successive minima:
\begin{equation}\label{vb8H}
\lambda^*_{\bx,i}(t):=e^{-\frac{t}{n+1}}\lambda_{\bx,i}(t)\,.
\end{equation}
Then,  \eqref{vb101B} becomes
$$
\prod_{i=1}^{n-d}\lambda^*_i\big(t,\bx)\le 1
$$
and \eqref{vb110A} and \eqref{vb110B}  respectively can be rewritten as
 \begin{align}
\|\bx'\wedge\bX\| &\le (n-d)!(1+\|\bx\|_1)\underbrace{\left(\prod_{i=1}^{n-d}\lambda^*_{\bx,i}(t) \right)}_{\hspace*{3ex}\le 1} e^{-\frac{(d+1)t}{n+1}},\label{vb110A+}\\[3ex]
\|\bX\|&\le (n-d)!(1+\|\bx\|_1)\underbrace{\left(\prod_{i=1}^{n-d}\lambda^*_{\bx,i}(t) \right)}_{\hspace*{3ex}\le 1}e^{\frac{(n-d)t}{n+1}}\label{vb110B+}\,.
\end{align}
Now, given $ N \in \N$  there exists $t \in \R  $  such that
\begin{equation}\label{vb_t}
(n-d)! (1+\|\bx\|_1) e^{\frac{(n-d)t}{n+1}}=N  \, .
\end{equation}
Clearly, if  $N$ is sufficiently large we can guarantee that $t>0$ and the following statement is a direct consequence of Lemma~\ref{lemB2A}.

\begin{prop}\label{lemB3A}
Let $n \in \NN$ and $d $ be an integer satisfying $ 0 \le d \le n-1$. Then for any  $\bx\in \RR^n$ and any
$$
N>(n-d)!(1+\|\bx\|_1) \, ,
$$
the decomposable vector $\bX\in\bigwedge^{d+1}\Z^{n+1}\setminus\{0\}$ associated with Lemma~\ref{lemB2A} with $t$ given by \eqref{vb_t},  satisfies
 \begin{equation*}
        \|\bx'\wedge\bX\|  \le c(n,d,\bx)\, N^{-\omega_d} \label{vb111A}  \qquad \text{and} \qquad \|\bX\| \le N \, ,
      \end{equation*}
where
$$
c(n,d,\bx):=\Big((n-d)! (1+\|\bx\|_1)\Big)^{\frac{n+1}{n-d}}\,.
$$
\end{prop}

Clearly, the proposition  is  a version of Lemma~\ref{l-d-type} obtained via the algebraic approach of \S\ref{noway}. We have already seen that the latter leads to Theorem~\ref{Dir-d}   --  a Dirichlet type statement concerning the approximation of points $\bx \in \R^n$  by $d$-dimensional rational subspaces. The upshot of this is that Theorem~\ref{DirInterm}, which implies Proposition~\ref{lemB3A} can thus be viewed as an optimal Dirichlet type theorem for approximating points by rational subspaces.

\subsubsection{Bad, Singular and Dirichlet Improvable in light of Theorem~\ref{DirInterm}  \label{almostdone}}

For any given $\bx \in \R^n$, observe that during the course of  establishing Proposition~\ref{lemB3A} we only used the trivial fact that
$\prod_{i=1}^{n-d}\lambda^*_{\bx,i}(t) \le 1$. Clearly if the product under consideration  gets significantly smaller than $1$, then we will obtain a stronger version of Proposition~\ref{lemB3A}. Specifically, if for some $\eps>0$
\begin{equation}\label{vb8G}
\prod_{i=1}^{n-d}\lambda^*_{\bx,i}(t) \le \eps
\end{equation}
for all sufficiently large $t$, then for any sufficiently large $N$ there exists a decomposable $\bX\in\wedge^{d+1}\Z^{n+1}$  such that
\begin{align*}
\|\bx'\wedge\bX\| \le c(n,d,\bx)\,\eps^{\frac{n+1}{n-d}} N^{-\omega_d} \qquad \text{and} \qquad
0<\|\bX\|\le N  \, .
\end{align*}
This is obtained in exactly the same way as Proposition~\ref{lemB3A}  is obtained from  Theorem~\ref{DirInterm} but now  we are able to insert the factor $\eps$ into the left hand side of \eqref{vb_t}.
Now observe that  in view of \eqref{vb301} and \eqref{vb8H}, inequality \eqref{vb8G} is equivalent to the statement that
$$
\frac{(n-d)t}{n+1}-\sum_{i=1}^{n-d}L_{\bx,i}(t)\ge \log(1/\eps)=:\delta\,
$$
which is precisely inequality \eqref{equ-equiSING}. We have already seen that the validity of \eqref{equ-equiSING} for arbitrarily large $\delta$ for all sufficiently large $t$ is equivalent to $\bx$ being $d$-singular (see Lemma~\ref{l-bad-sing2}).  As a consequence,  we have that:
\begin{itemize}
  \item $\bx \in \Sing_n^d$ if and only if for any  $\eps > 0$ there exists a constant $t_0= t_0({\eps})  > 0 $ such that \eqref{vb8G} holds for all $t \ge  t_0$.
\end{itemize}

\noindent Now suppose that  for some $\eps>0$
\begin{equation}\label{vb8G+}
\prod_{i=1}^{n-d}\lambda^*_{\bx,i}(t) \ge \eps
\end{equation}
 for all sufficiently large $t$.   This is easily seen to be equivalent to the statement that
$$
\frac{(n-d)t}{n+1}-\sum_{i=1}^{n-d}L_{\bx,i}(t)\le \log(1/\eps)=:\delta
$$
which is precisely inequality \eqref{equ-equi}.  This together with Lemma~\ref{l-bad-sing2} implies that

\begin{itemize}
\item  $\bx\in\Bad^d_n$ if and only if there exists $\eps>0$ such that \eqref{vb8G+} holds for all $t > t_0(\bx, \eps)$ sufficiently large.
\end{itemize}

\begin{remark}
Note that the above discussion  together with  Proposition~\ref{prop-equiv} and Lemma~\ref{age}, implies that the $d$-singular sets (respectively the $d$-badly approximable sets) defined via the algebraic approach of \S\ref{noway} or the  geometric approach of this section  coincide with the set $\Sing_n$ (respectively $\Bad_n$); that is, the set of singular (respectively, badly approximable)  points in $\R^n$ defined via either the classical simultaneous or dual form of  Dirichlet's  Theorem.
\end{remark}

\noindent The above observations naturally lead us to the following  notion of $d$-Dirichlet improvable points in $\R^n$.

\begin{definition} \label{VBdef}
Let $n \in \NN$ and $d $ be an integer satisfying  $0 \leq d \leq n-1$.  Then for  $0<\eps<1$, let
\begin{itemize}
\item
$\widehat\DI_n^d(\eps)$ to be the set of $\bx\in \RR^n$, such that inequality \eqref{vb8G},
or equivalently
\begin{equation}\label{vb101D}
\frac{(n-d)t}{n+1}-\sum_{i=1}^{n-d}L_{\bx, i}(t)\ge \log(1/\eps),
\end{equation}
holds for all $t > t_0(\bx, \eps)$ sufficiently large.
\end{itemize}
Moreover, let
$$
\widehat\DI_n^d:=\bigcup_{0<\eps<1}\widehat\DI_n^d(\eps)\, .
$$
\end{definition}

\medskip

Note that by definition,  $\bx\in \widehat\DI_n^d$ if and only if there exists $\eps\in(0,1)$ such that \eqref{vb8G}, or equivalently \eqref{vb101D}, holds for all sufficiently large $t$. The following result shows that these sets are in fact all equivalent and coincide with the set $\DI_n$; that is, the set of Dirichlet improvable points in $\R^n$ defined via either the classical simultaneous or dual form of  Dirichlet's  Theorem (both forms give rise to the same set).   In turn,  this together with Proposition~\ref{prop-d-di} implies that the set of $d$-Dirichlet improvable points defined via the algebraic approach (cf. Definition~\ref{BLdef}) and geometric approach also coincide.

\begin{lemma}\label{LA7}

\begin{itemize}
\item[{\rm(i)}] For any  $\eps\in(0,1)$,  we have that
\begin{equation}\label{vb8T=}
\widehat\DI^{n-1}_n(\eps)=\DI_n(\eps^{n+1})  \, .
\end{equation}
Therefore,  $\widehat\DI^{n-1}_n=\DI_n$.
\item[{\rm(ii)}]
For any $\eps\in(0,1)$ we have that
\begin{equation}\label{vb8TUP}
\widehat\DI^d_n(\eps)\subset \widehat\DI^{d+1}_n\left(\eps^{\frac{n-d-1}{n-d}}\right)\qquad\text{if }~~0\le d\le n-2
\end{equation}
and
\begin{equation}\label{vb8TDOWN}
\widehat\DI^d_n(\eps)\subset \widehat\DI^{d-1}_n\left(\eps^{\frac{d}{d+1}}\right)\qquad\text{if }~~1\le d\le n-1\, .
\end{equation}
Therefore,  the sets $\widehat\DI^d_n$ are the same for all $d\in\{0,1,\dots,n-1\}$.
\end{itemize}
\end{lemma}

\begin{proof}
Part (i) follows immediately on comparing \eqref{vb101D} in Definition~\ref{VBdef} and \eqref{DS+} in Lemma~\ref{DS}. Regarding Part~(ii), first suppose that $\eps\in(0,1)$, $0\le d\le n-2$ and $\bx\in\widehat\DI^{d}_n(\eps)$.  Then, by definition, for all sufficiently large $t$ we have \eqref{vb8G}. Now on raising \eqref{vb8G} to the power $n-d-1$ and using the inequalities $\lambda^*_{\bx,1}(t)\le\dots\le \lambda^*_{\bx,n+1}(t)$ we obtain that
\begin{align*}
\eps^{n-d-1}&\ge \big(\lambda^*_{\bx,1}(t)\cdots\lambda^*_{\bx,n-d}(t)\big)^{n-d-1}
\ge\big(\lambda^*_{\bx,1}(t)\cdots\lambda^*_{\bx,n-d-1}(t)\big)^{n-d}\,.
\end{align*}
Hence
$$
\lambda^*_{\bx,1}(t)\cdots\lambda^*_{\bx,n-d-1}(t)\le \eps^{\frac{n-d-1}{n-d}}
$$
and we conclude that $\bx\in\widehat\DI^{d+1}_n(\eps^{\frac{n-d-1}{n-d}})$. This proves the `going-up' inclusion \eqref{vb8TUP}.

To prove the `going-down' inclusion \eqref{vb8TDOWN}, let $\eps\in(0,1)$ and $1\le d\le n-1$ and $\bx\in\widehat\DI^{d}_n(\eps)$. Then, for all sufficiently large $t$ inequality \eqref{vb8G} holds and  since $\lambda^*_{\bx,1}(t)\le\dots\le \lambda^*_{\bx,n+1}(t)$, it follows via  Minkowski's second convex body theorem that
$$
\lambda^*_{\bx,n-d+1}(t)^{d+1}\le \lambda^*_{\bx,n-d+1}(t)\cdots \lambda^*_{\bx,n+1}(t)\le \frac{1}{\lambda^*_{\bx,1}(t)\cdots \lambda^*_{\bx,n-d}(t)}\,.
$$
Hence, on using \eqref{vb8G} we find that
$$
\prod_{i=1}^{n-d+1}\lambda^*_{\bx,i}(t)\le \left(\prod_{i=1}^{n-d}\lambda^*_{\bx,i}(t)\right)^{1-\frac{1}{d+1}}\le \eps^{\frac{d}{d+1}}\,.
$$
Since this holds for all sufficiently large $t$, we conclude that $\bx\in\widehat\DI^{d-1}_n\left(\eps^{\frac{d}{d+1}}\right)$ and this thereby  establishes  \eqref{vb8TDOWN}. The `therefore' statements in both parts of the lemma follow on taking the union over all $\eps\in(0,1)$ in \eqref{vb8T=}, \eqref{vb8TUP} and \eqref{vb8TDOWN}.
\end{proof}

\subsection{Final remarks: linear forms and weighted approximation}\label{Sec_A4}

The geometric approach of \S\ref{DTTGN} can be readily adapted to extend the definition of badly, singular and Dirichlet improvable points to systems of linear forms  and thereby consider the ``matrix'' analogue of the main problem considered in this paper. We will now develop/describe  this ``matrix'' theory within the  more general setup of {\em weighted} Diophantine approximation. The parametric geometry of numbers, the underlying tool exploited in proving our main result,  for weighted approximation has  recently been initiated by Schmidt \cite{Snew},  albeit it remains in its `infancy' to cope with the weighted analogue of the main question considered in this paper.

Let $ m, n \in \N $ and fix a pair of weights $\br=(r_1,\dots,r_n)\in\R^n_{\ge0}$ and $\bs=(s_1,\dots,s_m)\in\R^m_{\ge0}$ satisfying
$$
r_1+\dots+r_n=1\qquad\text{and}\qquad s_1+\dots+s_m=1\,.
$$
For $ t > 0$, let
$$
g_t={\rm diag}\{e^{-r_1t},\dots,e^{-r_nt},e^{s_1t},\dots,e^{s_mt}\}
\quad \text{ and } \quad
u_X= \begin{pmatrix}
I_n & 0 \\
X & I_{m} \end{pmatrix}\,,
$$
where $X$ is an $m\times n$ real matrix; i.e. $X\in M_{m\times n}(\RR)$. For each $i=1, \ldots, m+n$ and $t >0$,  let
$$
\lambda^*_{X,i}(t):=\lambda_i\big(g_t u_X\Z^{m+n}, B_1\big)
$$
denote the $i$-th successive minima of the convex body $B_1=[-1,1]^{m+n}$ with respect to the lattice $g_t u_X\Z^{m+n}$.
The following provides a natural generalisation of the Diophantine sets $\Bad^d_n$, $\Sing^d_n  $ and $\DI^d_n$  to the weighted matrix setup.

\begin{definition}
Let $X\in M_{m\times n}(\RR)$ and $d $ be an integer satisfying $ 0 \le d \le  m+n-2$.   Then for any  pair of weights $(\br,\bs)  \in \R^n \times \R^m$, we say that
\begin{itemize}
  \item[]
$X\in\Bad^d(\br,\bs)$ \,$\iff$ $\liminf\limits_{t\to\infty}\prod_{i=1}^{n+m-1-d}\lambda^*_{X,i}(t)>0$\,;
\item[]
$X\in\Sing^d(\br,\bs)$ $\iff$ $\lim\limits_{t\to\infty}\prod_{i=1}^{n+m-1-d}\lambda^*_{X,i}(t)= 0$\,;
\item[]
$X\in\DI^d(\br,\bs)$ \,\,\,\,\,$\iff$ $\limsup\limits_{t\to\infty}\prod_{i=1}^{n+m-1-d}\lambda^*_{X,i}(t)<1$\,.
\end{itemize}
\end{definition}

\medskip

\noindent {It is easily seen that up to the linear change of the variable $$t\mapsto \frac{n}{n+1}t \, ,$$
the definition of $\lambda^*_{X,i}(t)$ for matrices is consistent with that of $\lambda^*_{\bx,i}(t)$ for vectors $\bx\in\R^n$ given by \eqref{vb8H}.
 So if we put $m=1$ and $r_1= \ldots =r_n = 1/n$, we immediately see, in view of the discussion in \S\ref{almostdone}, that sets  $\Bad^d(\br,\bs)$, $\Sing^d(\br,\bs)  $ and  $ \DI^d(\br,\bs)$ coincide with $\Bad^d_n$, $\Sing^d_n  $ and $\DI^d_n$ respectively.
Moreover,} on using the same techniques as in the proof of Lemma~\ref{LA7}, it is relatively straightforward to show that for any fixed pair of weights $(\br,\bs)$ the badly approximable sets $\Bad^d(\br,\bs)$ are all the same;
the singular sets $\Sing^d(\br,\bs)$ are all the same, and the Dirichlet improvable sets $\DI^d(\br,\bs)$ are all the same. It is therefore natural to drop the dependence on `$d$' from the defining notation. With this in mind, the following is the natural ``matrix''  extension of the main question considered in this paper.

\medskip

\begin{problem}\label{prob-victor}
Let $ m, n \in \N $ such that $mn>1$. Prove that for any pair of weights $(\br,\bs)  \in \R^n \times \R^m$,
$$
\FS(\br,\bs):=  \DI(\br,\bs)\setminus\big(\Bad(\br,\bs)\cup \Sing(\br,\bs)\big)   \neq \emptyset  \, .
$$
\end{problem}

\medskip

We suspect that a lot more is true and the set under consideration has full dimension; that is to say that  $\dim \FS(\br,\bs) =nm$. This is in line with the content of Problem~\ref{prob-HD} appearing in \S\ref{FC} of the main body.

\begin{remark}
Potentially, the non-weighted case of Problem~\ref{prob-victor} can be resolved using the generalisation of Theorem~\ref{thm-roy} to systems of linear forms due to Das, Fishman, Simmons and Urba\'nski \cite[Theorem~5.2]{VarPrinc}. {However, there is currently a natural  barrier in using their generalisation for this purpose.   The point is that if we apply    \cite[Theorem~5.2]{VarPrinc},  then the analogue of the constant appearing in the right hand side of \eqref{vb+} becomes dependent  on the matrix $X$ under consideration.  The upshot of this is that we are not able to simply implement the machinery used to address Problem~\ref{prob-sanju} in order to deal with the analogous problem for systems of linear forms.}
\end{remark}

\vspace*{2ex}

\noindent{\bf Acknowledgements.} SV would like to thank Dorothy and the Tin Man for eighteen glorious years.  It has been a privilege watching you turn into feisty, loving, young people with firm moral compasses.   May you have many many wonderful adventures somewhere over the rainbow! SV would also like to take this opportunity to thank Prafula and Chiman for their love and for always being around for their little brother! Finally, many thanks to Bridget who, like me, sees our differences (oh and there are many!) as our great strength.

\vspace{0ex}

\

{\small

\noindent Victor  Beresnevich:\\
 Department of Mathematics,
 University of York,\\
 Heslington, York, YO10 5DD,
 England\\
 E-mail: {\tt victor.beresnevich@york.ac.uk}

 \bigskip

 \noindent Lifan Guan:\\
 Mathematisches Institut,
 Georg--August Universit\"{a}t G\"ottingen, \\
Bunsenstr. 3–5, D–37073 G\"{o}ttingen,
Germany.\\
 E-mail: {\tt guanlifan@gmail.com}

 \bigskip

 \noindent Antoine Marnat:\\
 Fakultät für Mathematik, Universität Wien, Oskar-Morgenstern-Platz 1,  \\
 1090 TU Graz, Austria. \\
 E-mail: {\tt antoine.marnat@univie.ac.at}

\bigskip

\noindent Felipe A. Ram\'{\i}rez:\\
 Department of Mathematics and Computer Science,
 Wesleyan University,\\
265 Church Street
Middletown, CT 06459\\
 E-mail: {\tt framirez@wesleyan.edu}

\bigskip

\noindent Sanju L. Velani:\\
 Department of Mathematics,
 University of York,\\
 Heslington, York, YO10 5DD,
 England\\
 E-mail: {\tt sanju.velani@york.ac.uk}

}

\end{document}